\documentclass[11pt]{amsart}
\pdfoutput=1

\pdfpageattr{/Group <</S /Transparency /I true /CS /DeviceRGB>>}

\usepackage[T1]{fontenc}
\usepackage[english]{babel}
\usepackage{geometry}
\usepackage{amsfonts,amsmath,amssymb,amsthm,mathrsfs}
\usepackage{tikz}
\usetikzlibrary{calc,decorations.markings}
\usepackage{caption}
\usepackage{subcaption}
\usepackage{xspace}
\usepackage{booktabs,tabularx}
\usepackage{dcolumn}
\usepackage{marvosym}
\usepackage{microtype}
\usepackage{enumitem}
\usepackage{mathtools}
\usepackage{booktabs}
\usepackage{url}
\usepackage[hidelinks]{hyperref}
\usepackage{lmodern}

\microtypesetup{tracking,kerning,spacing}
\microtypecontext{spacing=nonfrench}

\newcommand\CC{{\mathbb C}}

\newcommand\CCt{{\mathbb C}\{\!\!\{t\}\!\!\}}

\newcommand\RR{{\mathbb R}}

\newcommand\PP{{\mathbb P}}
\newcommand\TT{{\mathbb T}}
\newcommand\TP{{\mathbb{TP}}}

\newcommand\smallSetOf[2]{\{{#1}\,|\,{#2}\}}

\DeclareMathOperator{\val}{val}

\DeclareMathOperator{\Gr}{Gr} 

\DeclareMathOperator{\rank}{rk}

\DeclareMathOperator{\trop}{trop}
\DeclareMathOperator{\Dil}{D}
\DeclareMathOperator{\supp}{supp}
\DeclareMathOperator{\Dr}{Dr}
\DeclareMathOperator{\rowspace}{rowspace}

\newcommand\sgn[3]{|#1\cap [#3]|+|#2\cap [#3]|}

\theoremstyle{plain}
    \newtheorem{theorem}{Theorem}
    \newtheorem{corollary}[theorem]{Corollary}
    \newtheorem{lemma}[theorem]{Lemma}
    \newtheorem{proposition}[theorem]{Proposition}
\theoremstyle{definition}
    \newtheorem{remark}[theorem]{Remark}
    \newtheorem{example}[theorem]{Example}
    \newtheorem{definition}[theorem]{Definition}

    \newtheorem{question}{Question}

\title[Moduli spaces of codimension-one subspaces in a linear variety]{Moduli spaces of codimension-one subspaces in\\ a linear variety and their tropicalization}

\author{Philipp Jell \and Hannah Markwig \and Felipe Rinc\'on \and Benjamin Schr\"oter}

\address[Philipp Jell]{
Fakult\"at Mathematik,  
Universit\"at Regensburg,
93040 Regensburg
Germany
}
\email{philipp.jell@ur.de}

\address[Hannah Markwig]{
  Eberhard Karls Universit\"at T\"ubingen,
  Fachbereich Mathematik,
  Auf der Morgenstelle 10,
  72076 T\"ubingen
}
\email{hannah@math.uni-tuebingen.de}

\address[Felipe Rinc\'on]{
  School of Mathematical Sciences, Queen Mary University of London, London, Great Britain.
}
\email{ f.rincon@qmul.ac.uk}

\address[Benjamin Schr\"oter]{
  Department of Mathematics, KTH Royal Institute of Technology, Stockholm, Sweden
}
\email{schrot@kth.se}

\begin{document}

\begin{abstract}
	We study the moduli space of $d$-dimensional linear subspaces contained in a fixed $(d+1)$-dimensional linear variety $X$, and its tropicalization. We prove that these moduli spaces are linear subspaces themselves, and thus their tropicalization is completely determined by their associated (valuated) matroids. We show that these matroids can be interpreted as the matroid of lines of the hyperplane arrangement corresponding to $X$, and generically are equal to a Dilworth truncation of the free matroid. In this way, we can describe combinatorially tropicalized Fano schemes 
	and tropicalizations of moduli spaces of stable maps of degree $1$ to a plane.
\end{abstract}


\maketitle

\section{Introduction}
\noindent The study of moduli spaces such as Grassmannians, flag varieties, and moduli spaces of curves has a long tradition in algebraic geometry. These moduli spaces and their tropical analogues are central in the active field of tropical geometry \cite{MS15}. In this article we study the geometry and combinatorics of the moduli space $L(X)$ of codimension-$1$ subspaces of a linear variety $X\subset\mathbb{P}^{n-1}$. We then use this to describe its tropicalization $\trop(L(X))$, whose points parametrize tropicalized linear spaces in $\mathcal X = \trop(X)$ that can be obtained as tropicalizations of codimension-$1$ linear subspaces of $X$. 

If $X$ is a $(d+1)$-dimensional linear space then $L(X)$ is a subvariety of the Grassmannian $\Gr(d,n) \subset \PP^{\binom{n}{d}-1}$ cut out by certain linear incidence relations (see Section \ref{sec:preliminaries}). We show in Theorem~\ref{thm-subspace} that $L(X)$ itself is a linear subspace of $\PP^{\binom{n}{d}-1}$ and we present explicit descriptions of it. Motivated by tropical geometry, we ask the following question.
\begin{question}\label{ques-mapmatroid}
Given a linear variety $X \subset \mathbb{P}^{n-1}$, what is the matroid corresponding to the subspace $L(X) \subset \PP^{\binom{n}{d}-1}$ of codimension-$1$ subspaces of $X$?
\end{question}

Theorem \ref{thm-main} answers this question by giving several characterizations of this matroid. In particular, we show that the matroid of $L(X)$ is the matroid of lines of the hyperplane arrangement corresponding to $X$ (see Section \ref{sec-matroid of lines}). Generically, this matroid is equal to a Dilworth truncation of the free matroid, as we prove in Section \ref{sec-genericity}.

A motivation for Question \ref{ques-mapmatroid} arises from the following observation when $X$ is $2$-dimensional. Classically, the dual projective plane $(\PP^2)^\vee$ parametrizes lines in $\PP^2$, offering the opportunity for a rich theory of duality for classical plane curves.
In tropical geometry, there are many models of a plane: all $2$-dimensional tropical linear spaces; see Section~\ref{chap-tropical} for further details. The basic model of a tropical plane, $\RR^2$, allows for a well-known concept of duality: the space of non-degenerate tropical lines in $\RR^2$ can be identified with $\RR^2$ itself, by sending a line to (minus) the coordinates of its vertex. (If we choose the minus sign, also point-line incidences are respected under this tropical map to the dual plane.)
The concept of dual tropical curves and tropical point-line geometry has been investigated in this setting in \cite{IL18, BJLR17}.

The theory of embedded planar tropical curves makes it evident that the basic model $\RR^2$ is not always sufficient to study certain geometric features on the tropical side -- sometimes we have to take other tropical planes into account, see e.g.\ \cite{BJMS14, HMRT18}.
If we pick another tropical model of the plane, the concept of duality might be different: for instance, the space of tropical lines in the standard tropical plane $\mathcal{X}$ in $\RR^3$ cannot naturally be identified with $\mathcal{X}$ itself (see Example \ref{ex-standardplaneR3}). Thus, self-duality does not necessarily hold for tropical planes.
This observation motivates us to ask the following question (also studied in \cite{La18}:

\begin{question}Given a tropicalized plane $\mathcal{X} \subset \RR^{n-1}$, what is the space of tropicalized lines in $\mathcal{X}$? More generally, given a tropicalized linear space $\mathcal{X} \subset \RR^{n-1}$, what is the space of tropicalized codimension-$1$ spaces in $\mathcal{X}$?\label{ques-linesinplanes}\end{question}

Question \ref{ques-linesinplanes} is intimately related to Question \ref{ques-mapmatroid}, as we now explain (see Section \ref{chap-tropical} for further details).
Since the locus $L(X)$ of codimension-$1$ subspaces of $X$ is a linear space in $\PP^{\binom{n}{d}-1}$, its tropicalization depends exclusively on its (valuated) matroid.
In the case where $X$ has constant coefficients, for instance, so does $L(X)$, and thus its tropicalization is equal to the Bergman fan of its associated matroid. 
Our results from Section~\ref{chap-matroids} allow us to characterize the tropicalization of $L(X)$ as the Bergman fan of a Dilworth truncation whenever $X$ is generic. However, our tools also apply to the non-constant coefficient case -- see Remark~\ref{rem-nonconstantcoefficients} and Example~\ref{ex:valuated}. 

The tropicalized moduli spaces that we investigate are closely related to other moduli spaces studied in the literature. For example, in some cases it can be viewed as a special case of a tropicalized moduli space of stable maps (of degree one), or as a tropicalized Fano scheme. Moreover, the study of the tropicalizations of $L(X)$  can be useful in the investigation of tropicalizations of flag varieties by allowing the space $X$ to vary (see Subsection \ref{subsec-flag}). 
We explain these relations in further detail in Section \ref{sec-background}.

In summary, this paper is organized as follows. The main body is divided into two parts, Section \ref{chap-matroids} and \ref{chap-tropical}, which do not require any knowledge about tropical geometry. 

The first part, Section \ref{chap-matroids}, is focused on answering Question \ref{ques-mapmatroid} and deals with the matroid of the locus $L(X)$ of codimension-one subspaces of $X$. In Subsection \ref{sec:preliminaries}, we 
recall facts about Pl\"ucker coordinates, Grassmannians, and incidence relations. 
In Subsection \ref{sec-subspaceofcodimonespaces}, we show in Theorem \ref{thm-subspace} that the moduli space $L(X)$ is itself a linear subspace. The short argument we give there builds merely on linear algebra. In Theorem \ref{thm-incidenceimpliesPluecker} and its proof in the Appendix, we present a deeper algebraic argument that shows how the ideals of the incidence relations and the Pl\"ucker relations are related. 
In Subsection \ref{sec-matroid of lines}, we answer Question \ref{ques-mapmatroid} in Theorem \ref{thm-main}, providing different perspectives on the matroid of $L(X)$. In Subsection \ref{sec-genericity} we show that, in the case where $X$ is generic, the matroid of $L(X)$ can be described as a Dilworth truncation of the free matroid.

The second part, Section \ref{chap-tropical}, deals with the tropicalization of the moduli space $L(X)$ of codimension-$1$ subspaces in $X$. In Subsection \ref{subsec-tropicalprelim}, we introduce preliminaries on tropicalizations and tropical linear spaces. In Subsection \ref{subsec-tropicalization}, we build on Theorem \ref{thm-main} to study tropicalizations of spaces of subspaces, thus giving a partial answer to Question \ref{ques-linesinplanes}. We present various examples enlightening the different perspectives. In Subsection \ref{sec-background}, we discuss connections to other tropicalizations of moduli spaces which have been studied in the literature.

Finally, as mentioned before, in the Appendix we prove a special relation between the ideals generated by the Pl\"ucker relations and the codimension-$1$ incidence relations, which is formulated in Theorem~\ref{thm-incidenceimpliesPluecker}.

 \subsection{Acknowledgements}
 This project was motivated by an inspiring talk by Dhruv Ranganathan -- many thanks for the talk.
 We would like to thank Lara Bossinger, Ghislain Fourier, Marvin Hahn, Sara Lamboglia, and Dhruv Ranganathan for helpful discussions.
 We also want to thank two anonymous referees for their very careful reading and their comments which helped us improve an earlier version of this paper, in particular for pointing out a simplified proof of the equality of matroids considered below.
The second author gratefully acknowledges support by DFG-collaborative research center TRR 195, Project-ID 286237555.
The third author thankfully acknowledges partial support by the Research Council of Norway grant 239968/F20.
Part of this work was completed during the the Mittag-Leffler program \emph{Tropical geometry, amoebas and polytopes} in Spring 2018, where all authors were in residence. We would like to thank the institute for their hospitality and excellent working conditions.

 \subsection{Index of notation} For convenience of the reader we provide a table of notation. 
\[
\begin{array}{cl}

X & \mbox{The $(d+1)$-dimensional linear space in which we consider subspaces}\\
q_I & \mbox{The Pl\"ucker coordinates of the linear space $X$}\\
R_{A,B} & \mbox{The Pl\"ucker relation for $A \in \binom{[n]}{m-1}$ and $B \in \binom{[n]}{m+1}$}\\
\Gr(m,n) & \mbox{The Grassmannian parametrizing $m$-dimensional linear spaces in $K^n$}\\
L(X) & \mbox{The moduli space of $d$-dimensional subspaces contained in $X$}\\
P_J& \mbox{Variables for the Pl\"ucker coordinates of a linear space $Y$ contained in $X$}\\
I_{A,B} & \mbox{The incidence relation for $A \in \binom{[n]}{d-1}$ and $B \in \binom{[n]}{d+2}$, see Equation (\ref{def-incidence})}\\
W& \mbox{A matrix whose rowspace is $X$}\\
V& \mbox{A matrix with certain minors of $W$ as entries, see Equation (\ref{defV})}\\
U& \mbox{The matrix of incidence relations, see Equation (\ref{defU})}\\
\mathcal{X}=\trop(X) & \mbox{The tropicalization of $X$, see Definition \ref{def-trop}}
\end{array}
\]

\section{The locus of codimension-1 subspaces and its matroid} \label{chap-matroids}

\subsection{Preliminaries}\label{sec:preliminaries}\noindent
We start by recalling some basic facts about Pl\"ucker coordinates, Grassmannians, and incidence relations. 
A more detailed exposition about these topics can be found in, for example, \cite{GriffithsHarris}. 

Let $m \leq n$ be non-negative integers. We will use the notation $[n]\coloneqq\{1,\dotsc,n\}$, $2^{[n]}$ for the power set of $[n]$, and $\binom{[n]}{m}\coloneqq\{A \in 2^{[n]} \mid \left|A\right| = m\}$. If $A \subset [n]$ and $i \in [n]$, we also abbreviate 
$A \cup i := A \cup \{i\}$ and $A \setminus i := A \setminus \{i\}$.

Fix an algebraically closed field $K$. The \emph{Grassmannian} $\Gr(m,n)$ is a projective algebraic variety parametrizing all $m$-dimensional linear subspaces of $K^n$, or equivalently, all $(m-1)$-dimensional projective subspaces of $\mathbb{P}^{n-1}$. It can be embedded into projective space $\PP^{\binom{n}{m}-1}$ via its \emph{Pl\"ucker embedding}: An $m$-dimensional subspace $L \subset K^n$ presented as the rowspace of a full-rank matrix $M \in K^{m \times n}$ corresponds to its vector
of \emph{Pl\"ucker coordinates} 
\[
\big(\det(M^I)\big)_I\in \PP^{\binom{n}{m}-1},
\]
where $I$ varies over all $m$-subsets of $[n]$ and $M^I$ denotes the maximal square submatrix of $M$ consisting of the columns indexed by $I$.

Consider the polynomial ring $S = K[P_I: I \in \tbinom{[n]}{m}]$. 
Under its Pl\"ucker embedding, the Grassmannian $\Gr(m,n) \subset \PP^{\binom{n}{m}-1}$ is 
the zero locus of the homogeneous prime ideal of $S$,
called the \emph{Pl\"ucker ideal},
generated by the quadratic Pl\"ucker relations: 
For any $A \in \binom{[n]}{m-1}$ and $B \in \binom{[n]}{m+1}$, the \emph{Pl\"ucker relation} $R_{A,B} \in S$ is
\[ R_{A,B} \coloneqq \sum_{i \in B \setminus A} (-1)^{|[i]\cap A| + |[i] \cap B|} P_{A \cup i} \, P_{B \setminus i}\enspace. \]

Now, let $d \leq e$ be nonnegative integers, and suppose $X$ is a fixed $e$-dimensional linear subspace of $K^n$, with Pl\"ucker coordinates $q_I$ for $I \in \tbinom{[n]}{e}$. A $d$-dimensional linear subspace $L$ of $K^n$ with Pl\"ucker coordinates $P_J$ for $J \in \tbinom{[n]}{d}$ is contained in $X$ if and only if the following linear incidence relations are equal to zero:
For any $A \in \tbinom{[n]}{d-1}$ and $B \in \tbinom{[n]}{e+1}$, 
the \emph{incidence relation} $I_{A,B}$ is
\begin{equation}\label{def-incidence}
 I_{A,B} \coloneqq \sum_{i \in B \setminus A} (-1)^{|[i]\cap A| + |[i] \cap B|} q_{B \setminus i} \, P_{A \cup i}\enspace .
\end{equation}

\subsection{The locus of codimension-one subspaces}\label{sec-subspaceofcodimonespaces}

Suppose $X$ is a $(d+1)$-dimensional linear subspace of $K^n$.
Fix a matrix $W \in K^{(d+1)\times n}$ such that
$$X = \rowspace(W).$$ 
For $i \in [d+1]$ and $J \in \binom{[n]}{d}$, let $W_{-i}^J$ be the square submatrix of $W$
whose rows are indexed by $[d+1] \setminus \{i\}$ and whose columns are indexed by $J$.
We define the $(d+1)\times \binom{n}{d}$ matrix $V$ via
\begin{equation}\label{defV}
V_{i,J}=(-1)^{i+1}\det(W_{-i}^J)\enspace .
\end{equation}
The rowspace of $V$ encodes the codimension-1 subspaces of $X$, as the following theorem shows. 

\begin{theorem}\label{thm-subspace}
Let $X \subset K^{n}$ be a $(d+1)$-dimensional linear subspace, 
and $L(X) \subset \Gr(d,n) \subset \PP^{\binom{n}{d}-1}$ be the moduli space of $d$-dimensional linear subspaces contained in $X$. 
Then $L(X)$ is itself a $d$-dimensional projective subspace of $\PP^{\binom{n}{d}-1}$. 
Indeed, if $X$ is the rowspace of a matrix $W$, then $L(X)$ is equal to the rowspace of the matrix $V \in K^{(d+1)\times \binom{n}{d}}$ (depending on $W$)
defined in Equation \eqref{defV} above.
\end{theorem}   

\begin{proof}
Any $d$-dimensional subspace $Y$ of $X$ is equal to the rowspace of a matrix $T \cdot W$ 
where $T \in K^{d \times (d+1)}$. By the Cauchy-Binet Formula, 
the maximal minors of $T \cdot W$ satisfy
\[
\det\big((T \cdot W)^J\big) = \sum_{i=1}^{d+1} \det\big(T^{[d+1]\setminus i}\big)\, \det\big(W_{-i}^J\big)\enspace ,
\]
where $T^{[d+1]\setminus i}$ denotes the maximal submatrix of $T$ obtained by deleting column $i$.
If $t \in K^{d+1}$ denotes the vector with coordinates $t_i = (-1)^{i+1} \det\big(T^{[d+1]\setminus i}\big)$, 
it follows that the vector $y \in \Gr(d,n)$ of Pl\"ucker coordinates of $Y$ satisfies
$y = t \cdot V$, showing that $y$ belongs to the rowspace of $V$.
Moreover, as $T$ varies over all full rank matrices in $K^{d \times (d+1)}$, 
its maximal minors vary over all of $\Gr(d,d+1) \cong \mathbb{P}^{d}$, 
and so $t$ also varies over all of $K^{d+1}$, 
which shows that $L(X)$ is in fact equal to the rowspace of $V$.

Finally, to see that $L(X)$ is a $d$-dimensional projective subspace of $\PP^{\binom{n}{d}-1}$, 
we note that the matrix $V$ has full rank $d+1$. 
Indeed, since $X$ is a $(d+1)$-dimensional linear subspace, the matrix $W$ contains an invertible
maximal square submatrix $W^I$, indexed by a $(d+1)$-subset of columns $I \subset [n]$. 
The maximal square submatrix $V^I$ of $V$, whose columns we index by subsets of the form 
$I \setminus \{j\}$ where $j \in I$, has entries equal to the $d \times d$ minors of $W^I$ up to sign. 
In fact, this matrix can be obtained from the cofactor matrix of $W^I$ by negating every other column. 
Since the cofactor matrix of $W^I$ is invertible and negating columns preserves invertibility, this shows that the maximal square submatrix $V^I$ of $V$ is
invertible, and thus $V$ has full rank $d+1$.
\end{proof}

We now describe the subspace $L(X)$ in a different way, as the kernel of a matrix encoding the
incidence relations which are defined in Equation \eqref{def-incidence}. 
For $I \in \binom{[n]}{d+1}$, denote by $q_I$ the Pl\"ucker coordinate $q_I := \det\big(W^I\big)$ of the
subspace $X$. Let $U$ be the $\big(\binom{n}{d-1}\binom{n}{d+2} \times \binom{n}{d}\big)$-matrix whose rows encode the coefficients of the incidence relations. 
That is, for $A \in \tbinom{[n]}{d-1}$, $B \in \tbinom{[n]}{d+2}$, and $J \in \binom{[n]}{d}$,
the entry of $U$ in the row indexed by $(A,B)$ and the column indexed by $J$ is
\begin{equation}\label{defU}
 U_{(A,B),J} := \begin{cases}
(-1)^{|[i]\cap A| + |[i] \cap B|} q_{B\setminus i} & \text{if } J = A \cup i \text{ for some } i \in B\setminus A,\\
0 & \text{otherwise.}
\end{cases}
\end{equation}

\goodbreak

\begin{corollary}\label{cor-kernel}
Let $X \subset K^{n}$ be a $(d+1)$-dimensional linear subspace.
The projective subspace $L(X) \subset \Gr(d,n) \subset \PP^{\binom{n}{d}-1}$ consisting 
of all $d$-dimensional linear subspaces contained in $X$ is equal to the kernel of the matrix 
$U \in K^{\binom{n}{d-1}\binom{n}{d+2} \times \binom{n}{d}}$ defined in Equation \eqref{defU}.
\end{corollary}
\begin{proof}
The kernel of $U$ is the zero set in $\PP^{\binom{n}{d}-1}$ of all incidence relations 
$I_{A,B}$, which express the conditions on the Pl\"ucker coordinates of a $d$-dimensional subspace
for it to be contained in $X$.
This means that $\ker(U) \cap \Gr(d,n)$ is the $d$-dimensional projective subspace $L(X)$.
However, as we show next, the kernel of $U$ has (projective) dimension at most $d$, and so it must be equal to $L(X)$.

It is enough to prove that $U$ has at least $\binom{n}{d} -(d+1)$ independent rows to conclude that $\ker(U)$ is at most $d$ dimensional. 
Let $C\in\binom{[n]}{d+1}$ be such that $q_C\neq 0$. 
There are  $\binom{n}{d} -(d+1)$ sets $I$ such that $|I| = d$ and $I\not\subset C$. 
For every such $I$, let $i \in I \setminus C$, and consider the row $(A,B)$ of $U$ with $A=I\setminus i$ and $B=C\cup i$. 
The $I$-th entry of this row is $U_{(A,B),I} = \pm q_C\neq 0$. Furthermore, for $J\in \binom{[n]}{d}$ with $J\neq I$ and $|J\cap C| \geq |I\cap C|$ we have $U_{(A,B),J} = 0$, as $J = A\cup j = (I\setminus i)\cup j$ for some $j\in B = C\cup i$ implies $|J\cap C| < |I\cap C|$ or $i=j$, and hence $I=J$. 
	This proves that, up to permutation, $U$ contains a $(\tbinom{n}{d} -d-1)\times(\tbinom{n}{d} -d-1)$ lower triangular matrix with diagonal entries equal to $\pm q_C \neq 0$.
\end{proof}

The fact that the subspace defined by the incidence relations $I_{A,B}$ lies inside the Grassmannian $\Gr(d,n)$ implies that the incidence relations of codimension-$1$ imply the Pl\"ucker relations (up to radical). In fact, a stronger algebraic statement holds. We could not find this result in the literature, thus for the sake of completeness we state it here and provide a proof in the appendix.

\begin{theorem}\label{thm-incidenceimpliesPluecker}
Consider the values $q_I$ in the incidence relations as variables (see Equation \eqref{def-incidence}). Then the Pl\"ucker relations are in the saturation of the ideal generated by the incidence relations of codimension-$1$. More precisely, for any fixed $C\in\tbinom{[n]}{d+1}$ we have $R_{A,B} \in (\mathcal{I}  : \langle q_C\rangle^\infty)$, where $\mathcal{I}$ is the ideal generated by all in incidence relations $I_{D,E}$ with $D,E\subset [n]$, $|D|=d-1$, and $|E|=d+2$.
\end{theorem}


We conclude this section with an example showing that the fact that the incidence relations imply the Pl\"ucker relations is a phenomenon unique to the codimension-1 case.

\begin{example}
The hypothesis that $e=d+1$ is essential for the results in this section. 
Consider the case $n=4$, $d=2$ and $e=4$, for instance,
there are no incidence relations (since any $2$-dimensional linear subspace is contained in the only $4$-dimensional linear subspace of $K^4$),  
and so clearly the Pl\"ucker relation defining $\Gr(2,4)$ is not implied by the incidence relations.
Also, the space of $2$-dimensional subspaces contained in $K^4$ is not a linear subspace of $\PP^5$,
it is equal to the Grassmannian $\Gr(2,4)$.

As a different example, take $n=5$, $d=2$ and $e=4$. 
It can be computationally verified that, in this case, the ideal $\mathcal I$ generated by the incidence relations
contains no polynomials involving only the variables $\{P_J\}_{J \in \binom{5}{2}}$, 
not even after saturating by all the variables $q_I$, and so the Pl\"ucker relations defining $\Gr(2,5)$ are not contained in the saturation of $\mathcal I$.
\end{example}

\subsection{The matroid of codimension-1 subspaces}\label{sec-matroid of lines}

\noindent We now present various ways of describing the matroid of the subspace 
$L(X) \subset \PP^{\binom{n}{d}-1}$ parametrizing codimension-$1$ subspaces of a linear subspace $X \subset K^n$. 
We start by recalling the concept of {matroid of lines} of a hyperplane arrangement.

\begin{definition}\label{def-matroidoflines}
Let $\mathcal A = \{H_k\}_{k=1}^n$ be a hyperplane arrangement in $K^{d+1}$. For $J \in \binom{[n]}{d}$, let 
\[
\ell_J \coloneqq \begin{cases}
\bigcap_{k \in J} H_k  & \text{if this intersection has dimension 1}, \\
0 & \text{otherwise}.
\end{cases}
\]
The \emph{matroid of lines} of $\mathcal A$ is defined as the matroid on the ground set $\binom{[n]}{d}$ that encodes the dependencies among the lines $\ell_J$. Here a collection of lines is independent if the dimension of the smallest linear subspace containing these lines equals the number of lines. 
\end{definition}

If $X \subset K^n$ is a linear subspace of dimension $d+1$, 
the $n$ coordinate hyperplanes $\{x_i = 0\}$ of $K^n$ restrict to an arrangement $\mathcal A(X)$ 
of $n$ hyperplanes in $X \cong K^{d+1}$.
The matroid of $X$ is the matroid on the ground set $[n]$ 
where the rank of a subset $A \subset [n]$ is equal to the codimension in $X$
of the intersection $\bigcap_{i \in A} \{x_i=0\}$. 

The hyperplane arrangement $\mathcal A(X)$ can also be presented in the following way.
Let $W$ be a $(d+1) \times n$ matrix whose rowspace is equal to $X$. 
Denote by $w_k \in K^{d+1}$ the $k$-th column of $W$, for $k=1,\dots,n$. 
Let $H_k \subset K^{d+1}$ be the hyperplane consisting of all $x \in K^{d+1}$ such that 
$w_k \cdot x = 0$. 
Then $\mathcal A(X)$ is linearly isomorphic to the hyperplane arrangement $\{H_k\}_{k=1}^n$ in $K^{d+1}$. 

We now get to one of our main results.

\begin{theorem}\label{thm-main}
Let $X \subset K^n$ be a $(d+1)$-dimensional linear subspace, 
and $L(X) \subset \Gr(d,n) \subset \PP^{\binom{n}{d}-1}$ be the subspace of 
$d$-dimensional linear subspaces contained in $X$. 
The matroid of $L(X) \subset \PP^{\binom{n}{d}-1}$ is a rank-$(d+1)$ matroid that can be described in the following ways: 
\begin{enumerate}[label=(\roman*)]
	\item\label{item:colsV}
		as the matroid of dependencies among the columns of the matrix $V$ defined in Equation \eqref{defV},
	\item\label{item:dualcolsU} 
		as the dual matroid to the matroid of dependencies among the columns of the matrix $U$ defined in Equation \eqref{defU}, and
	\item\label{item:lines}
		as the matroid of lines of the hyperplane arrangement $\mathcal A(X)$ induced by $X$.
\end{enumerate}
\end{theorem} 
\begin{proof}
By Theorem \ref{thm-subspace}, the projective subspace $L(X)$ is equal to the rowspace of $V$, 
which proves description \ref{item:colsV}. By Corollary \ref{cor-kernel}, the subspace $L(X)$
is also equal to the kernel of the matrix $U$, proving \ref{item:dualcolsU}.
To prove description \ref{item:lines}, consider the hyperplane arrangement 
$\mathcal A(X) = \{H_k\}_{k=1}^n$ in $K^{d+1}$, where $H_k$ consists of all $x \in K^{d+1}$ 
orthogonal to the $k$-th column $w_k$ of $W$. For any $J \in \binom{[n]}{d}$, the subspace $\ell_J$ in 
Definition \ref{def-matroidoflines} is a line if and only if the $(d+1) \times d$ 
submatrix $W^J$ of $W$ consisting of the columns indexed by $J$ has rank $d$, 
which is the case precisely when it has a nonzero maximal minor. 
The maximal minors of the submatrix $W^J$ are 
(up to sign) the entries of the column $V_J$ of the matrix $V$,
so $\ell_J \neq \{0\}$ if and only if $V_J \neq 0$. 
In this case, the line $\ell_J$ is equal to the left kernel of $W^J$, and by Cramer's Rule,
$\ell_J$ is spanned by $V_J$. 
This, together with description \ref{item:colsV}, proves description~\ref{item:lines}. 
\end{proof}

\begin{remark}\label{rem-nonconstantcoefficients}
In Section \ref{chap-tropical} below, we consider the case where $K$ is a field with a non-Archimedean valuation. If the coefficients of the equations defining the linear space $X$ are not contained in a subfield on which the valuation is trivial, it is interesting to study the \emph{valuated matroid} of the linear subspace $L(X)$. In fact, a valuated analog of Theorem \ref{thm-main} holds, with the same proof. An example of this is given in Example \ref{ex:valuated} below.
\end{remark}

\begin{example}\label{ex:matroid_of_lines} Consider the $3$-dimensional space $X\subset K^4$ spanned by the rows of the matrix
\[ W =
\begin{pmatrix}
1&0&0&1\\
0&-1&0&0\\
0&0&1&1
\end{pmatrix} \enspace .
\]
Its Pl\"ucker vector is $(-1,-1,0,-1)$, listed in lexicographic order.
The moduli space $L(X)\subset \Gr(2,4) \subset \PP^5$ of planes in $X$ is equal to the rowspace of $V$ or the kernel $U$, where
\[
V = \begin{pmatrix}
0&0&0&1&1&0\\
0&-1&-1&0&0&1\\
1&0&0&0&-1&0
\end{pmatrix} \quad \text{ and } \quad
U = \begin{pmatrix}
0&1&-1&0&0&0\\
-1&0&0&1&-1&0\\
0&-1&0&0&0&1\\
0&0&-1&0&0&1
\end{pmatrix} \enspace .
\]
The maximal minors of $V$ are (up to sign) the Pl\"ucker coordinates of the linear space $L(X)$. In lexicographic order, they are
\[
(0,1,1,0, 1,1,0, 0,1, 1,   0,0,0, -1,0, 0,   -1,0, 0, 1) \enspace.
\]
The subsets that index the nonzero coordinates of this Pl\"ucker vector form the bases of the matroid $M$ of $L(X)$, which is equal to the matroid of the columns of $V$ and also the dual of the matroid of the columns of $U$.

We end this example by describing $M$ as a matroid of lines. The subspace $X$ is the vanishing set of the equation $x_1+x_3-x_4=0$. The four coordinate hyperplanes in $K^4$ induce a hyperplane arrangement $\mathcal{A}(X)$ in $X$. Eliminating the last coordinate projects the arrangement $\mathcal{A}(X)$ onto the three coordinate hyperplanes $H_i=\{x_i=0\}$ where $i\leq 3$ and $H_4=\{x_1+x_3=0\}$ in $K^3$.
By Theorem \ref{thm-main}, the matroid of lines of the arrangement $\{H_1, H_2, H_3, H_4\}$ is $M$. Figure~\ref{fig:matroid_of_lines} shows this hyperplane arrangement in $K^3$ together with the six lines (three of them coincide), its intersection lattice, and the matroid $M$ as a point configuration.
\end{example}

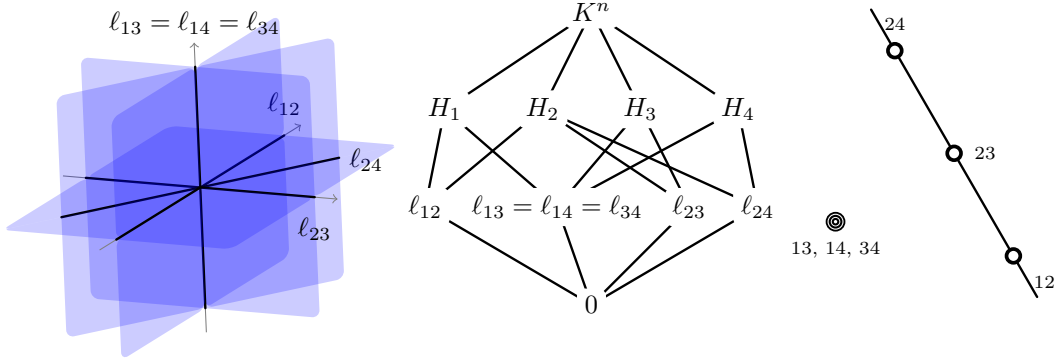
\begin{figure}[ht]
    \centering
    \begin{tikzpicture}[scale = 1.3,
                    line/.style={black, line width=0.9pt, line cap=round},
                    color = {black}]

\begin{scope}[x  = {(0.9cm,-0.076cm)},
              y  = {(-0.04cm,0.95cm)},
              z  = {(0.66cm,0.41cm)},
              scale = 1.3,
              face/.style={draw=none, opacity=0.2, fill=blue,rounded corners=3pt},
              axis/.style={draw=gray,->},
              color = {black}]

\coordinate (a0) at ( 1,  1, 0);
\coordinate (a1) at (-1, -1, 0);
\coordinate (a2) at ( 1, -1, 0);
\coordinate (a3) at (-1,  1, 0);

\coordinate (b0) at (0,  1,  1);
\coordinate (b1) at (0, -1, -1);
\coordinate (b2) at (0,  1, -1);
\coordinate (b3) at (0, -1,  1);

\coordinate (c0) at ( 1, 0,  1);
\coordinate (c1) at (-1, 0, -1);
\coordinate (c2) at ( 1, 0, -1);
\coordinate (c3) at (-1, 0,  1);

\coordinate (d0) at (-.7, -1, -.7);
\coordinate (d2) at ( .7, -1,  .7);
\coordinate (d3) at (-.7,  1, -.7);
\coordinate (d1) at ( .7,  1,  .7);

\draw[face] ($.5*(b0)+.5*(b3)$) -- (0,0,0) -- ($.5*(b1)+.5*(b3)$) -- (b3) -- cycle;
\draw[face] (c0) -- (0,0,0) -- ($.5*(c1)+.5*(c3)$) -- (c3) -- cycle;
\draw[axis] (0,0,0) -- (0,0,1.2);
\node[label = above:{\small $\ell_{12}$}]  at (0,0,1) {};
\draw[line] (0,0,0) -- ($.5*(b0)+.5*(b3)$);
\draw[face] ($.5*(b3)+.5*(b0)$) -- (0,0,0) -- ($.5*(b2)+.5*(b0)$) -- (b0) -- cycle;
\draw[face] (d2) -- ($.5*(d2)+.5*(d0)$) -- ($.5*(d3)+.5*(d1)$) -- (d1) -- cycle;
\draw[line] (0,0,0) -- ($.5*(d1)+.5*(d2)$);
\node[label = right:{\small $\ell_{24}$}]  at (.65,0,.65) {};
\draw[face] (0,0,0) -- ($.5*(c2)+.5*(c0)$) -- (c0) -- cycle;
\draw[face] (a0) -- (a2) -- (a1) -- (a3) -- cycle;
\draw[axis] (-1.2,0,0) -- (1.2,0,0);
\draw[axis] (0,-1.2,0) -- (0,1.2,0);
\draw[line] ($.5*(a0)+.5*(a3)$) -- ($.5*(a1)+.5*(a2)$);
\draw[line] ($.5*(a0)+.5*(a2)$) -- ($.5*(a1)+.5*(a3)$);
\draw[face] (0,0,0) -- ($.5*(c1)+.5*(c3)$) -- (c1) -- cycle;
\draw[face] (d0) -- ($.5*(d2)+.5*(d0)$) -- ($.5*(d3)+.5*(d1)$) -- (d3) -- cycle;
\draw[line] (0,0,0) -- ($.5*(d0)+.5*(d3)$);
\draw[face] ($.5*(b2)+.5*(b1)$) -- (0,0,0) -- ($.5*(b1)+.5*(b3)$) -- (b1) -- cycle;
\draw[draw=gray] (0,0,-1.2) -- (0,0,0);
\draw[face] (c2) -- ($.5*(c0)+.5*(c2)$) -- (0,0,0) -- (c1) -- cycle;
\draw[face] ($.5*(b1)+.5*(b2)$) -- (0,0,0) -- ($.5*(b2)+.5*(b0)$) -- (b2) -- cycle;
\draw[line] (0,0,0) -- ($.5*(b1)+.5*(b2)$);

\node[label = below:{\small $\ell_{23}$}]  at (1,0,0) {};
\node[label = above:{\small $\ell_{13}=\ell_{14}=\ell_{34}$}]  at (0,1.1,0) {};
\end{scope}


\coordinate (H1) at (2.5,.8);
\coordinate (H2) at (3.5,.8);
\coordinate (H3) at (4.5,.8);
\coordinate (H4) at (5.5,.8);
\coordinate (l12) at (2.3,-.2);
\coordinate (lx) at (3.65,-.2);
\coordinate (l23) at (5,-.2);
\coordinate (l24) at (5.7,-.2);
\coordinate (b) at (4,-1.2);
\coordinate (t) at (4,1.8);

\tikzstyle{node}=[text=black, inner sep=3pt, rectangle, rounded corners=3pt,fill=white, draw=none]

\foreach \i/\k in {1/12,1/x,2/12,2/23,2/24,3/23,3/x,4/24,4/x} {
   \draw[line] (H\i) -- (l\k);
  }

\foreach \i in {1,2,3,4} {
   \draw[line] (H\i) -- (t);
  }
  
\foreach \k in {12,23,24,x} {
   \draw[line] (l\k) -- (b);
  }

\node[node]  at (b) {\small $0$};
\node[node]  at (l12) {\small $\ell_{12}$};
\node[node]  at (lx) {\small $\ell_{13}=\ell_{14}=\ell_{34}$};
\node[node]  at (l23) {\small $\ell_{23}$};
\node[node]  at (l24) {\small $\ell_{24}$};
\node[node] at (H1) {\small $H_1$};
\node[node] at (H2) {\small $H_2$};
\node[node] at (H3) {\small $H_3$};
\node[node] at (H4) {\small $H_4$};
\node[node] at (t) {\small $K^n$};


  \tikzstyle{vertex}=[draw=black, fill=white, line width=1.4pt]
  \coordinate (trans) at (7.1,0);

  \coordinate (p24) at ($(90:1.4)+(trans)$); 
  \coordinate (p12) at ($(330:1.4)+(trans)$); 
  \coordinate (p23) at ($.5*(90:1.4)+.5*(330:1.4)+(trans)$); 
  \coordinate (p13) at ($(210:.7)+(trans)$); 

  \node[label=below right:{\tiny $12$}]  at (p12) {};
  \node[label=below:{\tiny $13$, $14$, $34$}]  at (p13) {};
  \node[label=right:{\tiny $23$}]  at (p23) {};
  \node[label=above:{\tiny $24$}]  at (p24) {};

  \draw[line] ($1.2*(p24)-.2*(p12)$) -- ($1.2*(p12)-.2*(p24)$);

  \draw[vertex] (p12) circle (2pt);
  \draw[vertex, line width=0.7pt] (p13) circle (2.6pt);
  \draw[vertex, line width=0.8pt] (p13) circle (1.7pt);
  \draw[vertex, line width=0.8pt] (p13) circle (0.8pt);
  \draw[vertex] (p23) circle (2pt);
  \draw[vertex] (p24) circle (2pt);
  
\end{tikzpicture}
    \caption{The hyperplane arrangement $\mathcal A(X)$ in $K^3$ (left), its intersection lattice (center), and the matroid $M$ of $L(X)$ (right), from Example~\ref{ex:matroid_of_lines}.}
    \label{fig:matroid_of_lines}
\end{figure}

\subsection{Genericity and Dilworth truncations} \label{sec-genericity}

\noindent The moduli space $L(X)$ and its matroid depend in general on the linear space $X$. In this subsection we show that the matroid of $L(X)$ is constant for generic spaces $X$, and we give a explicit combinatorial description of this matroid. 

\begin{proposition}\label{prop-generic}
There is an open dense subset (in the Zariski topology) of the Grassmannian $\Gr(d+1,n)$
such that for all $X$ in this subset the matroid of $L(X)$ is the same. 
We call a subspace $X$ in this subset \emph{generic}.
\end{proposition}

A word of warning on our terminology: the subset parametrizing generic subspaces is in general a strict subset of the set of linear subspaces whose matroid is uniform.

\begin{proof}
Assume $X$ is the rowspace of the matrix $W$.
Denote by $q_I := \det(W^I)$ the Pl\"ucker coordinates of $X$, for $I \in \binom{[n]}{d+1}$.
Assume without loss of generality that $q_{[d+1]} = 1$. We can row-reduce the matrix 
$W$ so that its submatrix $W^{[d+1]}$ is equal to the identity matrix.
In this situation, the entries of the matrix $V$ defined in Equation \eqref{defV} satisfy
$$V_{i,J}:=(-1)^{i+1}\det(W_{-i}^J)= \det(e_i \mid W^J),$$
where $(e_i \mid W^J)$ denotes the $(d+1) \times (d+1)$ matrix whose first column
is the coordinate vector $e_i$ and whose last $d$ columns are the submatrix $W^J$.
This shows that, in fact,
\begin{equation}\label{eq:entriesareplucker}
V_{i,J}= \begin{cases}
(-1)^{|[i] \cap J|}q_{J\cup\{i\}} & \text{ if } i \notin J,\\
0 & \text{ if } i \in J.
\end{cases}
\end{equation}
The maximal minors of $V$ are thus polynomials in the Pl\"ucker coordinates $q_I$.  
Some of these polynomials might be identically zero in the Grassmannian $\Gr(d+1,n)$. 
The desired open subset of $\Gr(d+1,n)$ is obtained by requiring that all the maximal minors of 
$V$ that are not identically zero do not vanish. 
Note that this results in an open dense subset of $\Gr(d+1,n)$, 
as this is an irreducible algebraic variety. 
\end{proof}

\begin{remark}\label{rem:reducible} It is interesting to study the matroid of the moduli space $L(X)$ also for non-generic $X$. In particular, given a realizable matroid $M$, one could ask what the matroid of $L(X)$ is for ``sufficiently generic'' $X$ realizing the matroid $M$. In the case where $M$ is not a uniform matroid, though, it is not even clear that the matroid of $L(X)$ is constant in a dense open set of the realization space of $M$, as this space might not be irreducible.
\end{remark}

\begin{example}\label{ex-u36}
	Let $d=2$ and $n=6$. After row-reduction, we may assume the matrix $W$ whose rows span the space $X$ is equal to
	\begin{align*}
		W = \begin{pmatrix}
			1&0&0&q_{234}&q_{235}&q_{236}\\
			0&1&0&-q_{134}&-q_{135}&-q_{136}\\
			0&0&1&q_{124}&q_{125}&q_{126}\\
		\end{pmatrix}.
	\end{align*}
	Using the lexicographic ordering for the $2$-sets of $\{1,2,\dots,6\}$, by Equation \eqref{eq:entriesareplucker} the $3 \times 15$ matrix $V$ from Equation \eqref{defV} is 
	\[
		V = \left(\begin{smallmatrix}
			0&0&0&0&0&1&q_{124}&q_{125}&q_{126}&q_{134}&q_{135}&q_{136}&q_{145}&q_{146}&q_{156}\\
			0&-1&-q_{124}&-q_{125}&-q_{126}&0&0&0&0&q_{234}&q_{235}&q_{236}&q_{245}&q_{246}&q_{256}\\
			1&0&-q_{134}&-q_{135}&-q_{136}&0&-q_{234}&-q_{235}&-q_{236}&0&0&0&q_{345}&q_{346}&q_{356}
		\end{smallmatrix}\right),
	\]
	where we used that $q_{123}=1$.

	The Pl\"ucker coordinates of $L(X)$ are given by the $3\times 3$ minors of the matrix $V$, 
	and the matroid of $L(X)$  encodes which of these minors vanish.
	Sixty out of the $\tbinom{15}{3}=455$ maximal minors of $V$ always vanish due to the fact that the $q_I$ satisfy the Pl\"ucker relations (see also Example~\ref{ex-n-2DWgeom} and Theorem~\ref{theorem-DW} below). Another $380$ of these minors are monomials in the $q_I$ when reduced modulo the Pl\"ucker relations.
	These minors are non-vanishing if the matroid of $X$ is the uniform matroid $U_{3,6}$, i.e., if all the Pl\"ucker coordinates $q_I$ are nonzero.
	The remaining $15$ maximal minors that do not reduce to monomials might vanish or not depending on $X$. 
	For example
	\[
		\det V^{\{12, 34, 56\}}= q_{134}\cdot q_{256} - q_{156}\cdot q_{234} = q_{124}\cdot q_{356} - q_{456},
	\]
	where the last equality follows from the Pl\"ucker relation $R_{56,1234}$, and
	\begin{align*}
		\det V^{\{16, 25, 34\}} &= q_{126}\cdot q_{134}\cdot q_{235} - q_{125}\cdot q_{136}\cdot q_{234}\\
					&=q_{134}\cdot\left(q_{125}\cdot q_{236} - q_{256} \right) - q_{125}\cdot\left(q_{134}\cdot q_{236}-q_{346} \right)\\ 
					&=q_{125}\cdot q_{346} -  q_{134}\cdot q_{256} \enspace .
	\end{align*}
	The $15$ maximal minors of $V$ that might vanish are obtained by permuting the labels $1,2,3$ and $4,5,6$. Up to signs, they are:
\begin{align*}
	\begin{array}{ccc}
		q_{124}\cdot q_{356}-q_{456}, & q_{125}\cdot q_{346}+q_{456}, & q_{126}\cdot q_{345}-q_{456},\\		
		q_{134}\cdot q_{256}+q_{456}, & q_{135}\cdot q_{246}-q_{456}, & q_{136}\cdot q_{245}+q_{456},\\		
		q_{234}\cdot q_{156}-q_{456}, & q_{235}\cdot q_{146}+q_{456}, & q_{236}\cdot q_{145}-q_{456},\\
		q_{125}\cdot q_{346}-q_{134}\cdot q_{256}, & q_{135}\cdot q_{246}-q_{234}\cdot q_{156}, & q_{235}\cdot q_{146}+q_{124}\cdot q_{356},\\
		q_{125}\cdot q_{346}+q_{234}\cdot q_{156}, & q_{135}\cdot q_{246}-q_{124}\cdot q_{356}, & q_{235}\cdot q_{146}-q_{134}\cdot q_{256}\;.
	\end{array}
\end{align*}
If we require that all Pl\"ucker coordinates $q_I$ and all these minors are nonzero, we obtain an open dense subset of $\Gr(3,6)$ parametrizing linear subspaces $X$ for which $L(X)$ is constant.
\end{example}


We now combinatorially describe the matroid of $L(X)$ in the case that the linear subspace $X$ is generic.
The relevant matroidal construction is called the Dilworth truncation.
The following is the definition from \cite[Chapter 7, Exercise 7.55]{White:1986} by Brylawski.

\begin{definition}\label{def-Dilworth}
Let $M$ be a rank-$e$ matroid on the ground set $[n]$.
For $1\leq k \leq e$, the $k$-th \emph{Dilworth truncation} $\Dil_k(M)$ of $M$ is 
defined as follows in terms of independent sets.
The ground set of $\Dil_k(M)$ is the set of flats of rank $k$ of $M$. 
The independent sets of $\Dil_k(M)$ are the sets $\{J_1,\ldots,J_m\}$ such that 
for any $s>0$ and $\{j_1,\ldots,j_s\}\subset [m]$, the union $\bigcup_{l=1}^s J_{j_l}$ has rank at least $s+k-1$ in $M$.
The matroid $\Dil_k(M)$ has rank $e-k+1$.
\end{definition}
Note that the first Dilworth truncation $\Dil_1(M)$ of a simple matroid $M$ is $M$ itself, and the $e$-th Dilworth truncation $\Dil_e(M)$ is the uniform matroid $U_{1,1}$ with one coloop. 

The following theorem gives a geometric interpretation of the Dilworth truncation of a matroid $M$ realizable over $\CC$.
\begin{theorem}[\cite{Bry85}]\label{thm-geomDW}
Let $P$ be a collection of $n$ points in $\mathbb{P}^{e-1}$ realizing the matroid $M$.
Suppose $H \subset \mathbb{P}^{e-1}$ is a generic subspace of codimension $k-1$.
Each $k$-flat $F$ spanned by $P$ (of projective dimension $k-1$) intersects $H$ in a point that we label $h_F$. 
Then this point configuration $\{h_F\}$ in $H$ realizes the $k$-th Dilworth truncation $\Dil_k(M)$ of $M$. 
\end{theorem}

We will prove that the matroid of $L(X)$ for generic linear subspaces $X$ is given by a certain relabeling of a Dilworth truncation of the free matroid $U_{n,n}$.

\begin{definition}\label{def-relabel} The matroid $\widetilde{\Dil}_k(U_{n,n})$ is the rank $n-k+1$ matroid on the set of subsets of $[n]$ of size $n-k$ obtained by relabeling the elements of the Dilworth truncation $\Dil_k(U_{n,n})$ with their complements. 
Concretely, a collection of $(n-k)$-subsets $J_1, \dots, J_m$ is independent in 
$\widetilde{\Dil}_k(U_{n,n})$ if and only if for any $s>0$ and $\{j_1,\ldots,j_s\}\subset [m]$
we have $|\bigcap_{l=1}^s J_{j_l}| \leq n-k+1-s$.
\end{definition}

\begin{example}\label{ex-n-2DWgeom}
We describe the circuits of the relabeling $\widetilde \Dil_{n-2}(U_{n,n})$. 
The elements of $\widetilde \Dil_{n-2}(U_{n,n})$ are subsets of $[n]$ of size $2$. 
In $\widetilde \Dil_{n-2}(U_{n,n})$, there are no $1$- or $2$-element circuits. 
A $3$-element subset $\{J_1,J_2,J_3\}$ is a circuit if and only if 
these 3 subsets have the form $\{a,a_1\}$, $\{a,a_2\}$, $\{a,a_3\}$. 
A $4$-element subset is a circuit if and only if it does not contain a $3$-element circuit as above. There are no circuits of size bigger than $4$. This matroid can be represented by the affine point configuration obtained by pairwise intersecting $n$ generic lines in $\PP^2$. 

For example, in the case $n=4$, an affine point configuration representing the matroid $\widetilde \Dil_{2}(U_{4,4})$ is pictured on the left of Figure \ref{fig-DW} in Subsection~\ref{subsec-tropicalization} below, together with the corresponding lattice of flats on the right. 
Note that this matroid is the (relabeled) graphical matroid of the complete graph on four nodes. In general, the Dilworth truncation $\Dil_{2}(U_{n,n})$ is the graphical matroid of the complete graph on $n$ nodes. 
\end{example}

Note that, even though the (relabeled) Dilworth truncation $\widetilde \Dil_{k}(U_{n,n})$ is a paving matroid when $k = n-2$, i.e., all its circuits have size at least $\rank(\widetilde \Dil_{k}(U_{n,n}))$, this is not the case for lower values of $k$. For example, the three subsets $[n] \setminus \{1,2\}$, $[n] \setminus \{1,3\}$, and $[n] \setminus \{2,3\}$ form a circuit in $\widetilde \Dil_2(U_{n,n})$, while $\rank(\widetilde \Dil_{2}(U_{n,n})) = n-1$.

The following theorem describes combinatorially the matroid of $L(X)$ in the generic case. 
This result also appears in \cite{AB07}, in connection to matroids arising from generic flag arrangements. 

\begin{theorem}\label{theorem-DW}
If $X \in \Gr(d+1,n)$ is a generic linear subspace
(see Proposition~\ref{prop-generic}), then the matroid of $L(X)$ equals the relabeling $\widetilde \Dil_{n-d}(U_{n,n})$.
\end{theorem}

\begin{proof}
Let $X \subset \PP^{n-1}$ be a generic linear subspace of codimension $n-d-1$. 
The $n$ coordinate vectors $e_i$ in $\mathbb{P}^{n-1}$ realize the free matroid $U_{n,n}$. 
Consider the hyperplanes $H_i := \{x_i = 0\}$ of $\mathbb P^{n-1}$, for $i \in [n]$.
If $J \subset [n]$ has size $n-d$, the intersection of the $d$ hyperplanes $\{H_i\}_{i \notin J}$ is the $(n-d-1)$-dimensional linear span $F_J$ of $\smallSetOf{e_i}{i\in J}$. By Theorem \ref{thm-geomDW}, the points $h_J = X \cap F_J$ realize the $(n-d)$-th Dilworth truncation $\Dil_{n-d}(U_{n,n})$. Note that the point $h_J = X \cap (\bigcap_{i\notin J} H_i)$ corresponds to the line $\ell_{J^c}$ of the arrangement $\mathcal{A}(X)$, as in Definition \ref{def-matroidoflines}. It follows by Theorem \ref{thm-main} \ref{item:lines} that the matroid of $L(X)$ equals the relabeling $\widetilde{\Dil}_{n-d}(U_{n,n})$.
\end{proof}

In view of Theorem \ref{theorem-DW}, we see that Example \ref{ex-n-2DWgeom} characterizes the matroid of the space $L(X)$ of lines in a generic plane $X \subset \mathbb P^{n-1}$.

\section{Tropicalizing the locus of codimension-one subspaces} \label{chap-tropical}

\subsection{Preliminaries on tropical geometry}\label{subsec-tropicalprelim}
We start with some basic preliminaries on tropical geometry and, in particular, tropical linear spaces. Readers already familiar with these topics can jump directly to Subsection \ref{subsec-tropicalization}. 

Tropical geometry can be viewed as algebraic geometry over the max-plus semifield, where addition is taking the maximum and multiplication is classical addition. These operations can be thought of as large scale estimate of standard operations, in the sense of what is happening to the order of magnitude of a number. From that point of view, it comes as no surprise that tropical geometry can be considered as a degeneration of algebraic geometry. Algebraic varieties are degenerated to certain polyhedral complexes called tropical varieties. This degeneration process is referred to as \emph{tropicalization}. 

The fact that tropical varieties are polyhedral complexes enables us to view tropical geometry as a way to infuse combinatorial techniques and methods from convex geometry to the world of algebraic geometry, and vice versa. 
Of course, tropical geometry also poses many interesting intrinsic challenges.

To make use of tropical geometry as a tool in algebraic geometry, the so-called \emph{realizability problem} is essential. This is the question of whether a given tropical variety arises as the tropicalization of an algebraic variety. A variation is the \emph{relative realizability} problem:

\begin{question}[Relative realizability problem]
Given a pair of tropical varieties $\mathcal{Y}\subset \mathcal{X}$ and an algebraic variety $X$ tropicalizing to $\mathcal{X}$, does there exist a subvariety $Y\subset X$ tropicalizing to $\mathcal{Y}$?
\end{question}
We offer an approach to the relative realizability problem for linear subspaces by studying tropicalizations of the spaces $L(X)$ of codimension-1 subspaces of $X$: by understanding all tropicalizations of codimension-1 subspaces in $X$, we know which tropical codimension-1 linear spaces $\mathcal{Y}\subset\mathcal{X}=\trop(X)$ are relatively realizable.

Our study of tropicalizations of spaces $L(X)$ largely builds on  Section \ref{chap-matroids}, making use of the theory of tropicalization of linear spaces.
We work over a ground field $K$ which is algebraically closed and has a non-Archimedean valuation $\val$, e.g., the field of Puiseux series  $ \CCt$.
\begin{definition}\label{def-trop}
Let $X\subset \mathbb{P}^{k}$ be an algebraic variety. Then the tropicalization $\trop(X)$ of $X$ is defined to be
$$ \trop(X)=\overline{-\val(X\cap (K^{\ast})^k)} \subset \mathbb{R}^k,$$
the Euclidean closure of the image of $X$ under the coordinate-wise negative valuation map.
\end{definition}

It is one the miracles of tropical geometry that $\trop(X)$ is a polyhedral complex. This polyhedral complex reflects properties of the algebraic variety $X$. In general there are no criteria for determining the realizability of a polyhedral complex as a tropicalization. 

Tropical linear spaces are natural generalizations of the polyhedral complexes obtained as tropicalizations of linear spaces. Their structure is encoded by tropical Pl\"ucker vectors (also called valuated matroids), as we now explain.

Let $\TT \coloneqq (\RR \cup \{-\infty\}, \max, +)$ be the {tropical semiring}. A vector $p \in \TT^{\binom{[n]}{m}}$ is called a {\em tropical Pl\"ucker vector} of rank $m$ if it satisfies the (quadratic) \emph{tropical Pl\"ucker relations}: for any $A \in \binom{[n]}{m-1}$ and $B \in \binom{[n]}{m+1}$, i.e., the maximum
\begin{equation*}\label{eqplucker}
 \max_{i \in B \setminus A} (p_{A+i} + p_{B-i})
\end{equation*}
is attained at least twice (or is equal to $-\infty$). The support $\supp(p) \coloneqq \{ I \in \tbinom{[n]}{m} \mid p_I \neq -\infty \}$ of $p$ is the collection of bases of a matroid on the ground set $[n]$, called the {\em underlying matroid} of $p$. 

Let $p \in \TT^{\binom{[n]}{m}}$ be a tropical Pl\"ucker vector. For any $B \in \tbinom{[n]}{m+1}$, let $c_B \in \TT^n$ be defined by
\begin{equation*}\label{defcircuit}
 (c_B)_i \coloneqq 
 \begin{cases}
  p_{B-i} & \text{ if $i \in B$,} \\
  -\infty & \text{ otherwise.}
 \end{cases}
\end{equation*}
When $c_B$ is not equal to $(-\infty, \dotsc, -\infty) \in \TT^n$, any vector of the form $c_B + \lambda \cdot \textbf{1}$ with $\lambda \in \RR$ is called a {\em (valuated) circuit} of $p$, where $\textbf{1} \coloneqq (1,\dotsc,1) \in \RR^n$. Its support $\supp(c_B + \lambda \cdot \textbf{1}) = \supp(c_B) \coloneqq \{ i \in [n] \mid (c_B)_i \neq -\infty \}$ is a circuit of the underlying matroid of $p$. If $\mathcal C(p)$ denotes the collection of circuits of $p$, the {\em tropical linear space} corresponding to $p$ is \begin{equation*}
\mathcal{L}(p) \coloneqq \left\{ x \in \TT^n : \text{ for all } c \in \mathcal C(p), \, \max_{i \in [n]}(x_i + c_i) \text{ is attained at least twice}\right\}.
\end{equation*}
The tropical linear space $\mathcal{L}(p)$ is a pure $m$-dimensional polyhedral complex, and it determines the tropical Pl\"ucker vector $p$ up to global tropical multiplication by a scalar. It follows that the set of all $m$-dimensional tropical linear spaces is parametrized by the {\em Dressian} $\Dr(m,n)$,
which is the subset of tropical projective space $\TP^{\binom{[n]}{m}-1}$ consisting of all tropical Pl\"ucker vectors of rank $m$; see \cite{HerrmannJensenJoswigSturmfels,MS15,OlartePanizzutSchroeter} for further details. (Here, $\TP^{\binom{[n]}{m}-1}$ is obtained by identifying vectors in $\TT^{\binom{[n]}{m}}$ that differ by a multiple of $(1,\ldots,1)$.)

Tropical linear spaces play the role of linear subspaces in tropical geometry, and generalize tropicalizations of classical linear subspaces. As above, let $K$ be a field with a non-Archimedean valuation, and consider the $n$-dimensional vector space $ K^n$. Suppose $X$ is an $m$-dimensional linear subspace of $K^n$ with Pl\"ucker coordinates $P \in \PP^{\tbinom{[n]}{m}-1}$, and let $p \in \TP^{\tbinom{[n]}{m}-1}$ be the tropical Pl\"ucker vector obtained by taking the coordinatewise valuation of $P$. Then the tropicalization of the linear subspace $X \subset K^n$ is precisely the tropical linear space $\mathcal{L}(p)$. Tropical Pl\"ucker vectors arising as valuations of classical Pl\"ucker vectors are called {\em realizable}, and they form the {\em tropical Grassmannian} $\trop(\Gr(m,n)) \subset \Dr(m,n)$.

The tropicalization of a linear subspace $X$ can be described more easily if $X$ is defined over a subfield on which the valuation is zero, e.g.,   $\CC\subset \CCt$. We call this the \emph{constant coefficient case}. Then, the tropicalization of $X$ restricted to $\RR^n$ is a fan called the \emph{Bergman fan}.
It only depends on the underlying matroid of $X$ (see \cite{Stu02}, \S\;9.3). 

Suppose $X$ is given as the rowspace of a matrix~$A$.
The matroid of $X$ (i.e., the matroid of $A$) can be specified by its collection of circuits, which are the minimal sets arising as supports of linear
forms in $\ker(A)$ (equivalently, minimal sets arising as supports of linear forms vanishing on $X$). These are minimal sets $\{i_1,\ldots,i_r\}\subset \{1,\ldots,n\}$ such that the columns $a_{i_1},\ldots,a_{i_r}$ of $A$ are linearly dependent.
  
   Given a vector $u\in \mathbb{R}^n$, let $\mathcal{F}(u)$ denote the unique \emph{flag of subsets}
   \begin{displaymath}
     \emptyset=: F_0 \subsetneq F_1 \subsetneq \ldots \subsetneq
     F_{m-1}\subsetneq F_m:= \{1,\ldots,n\}
   \end{displaymath}
   such that
   \begin{displaymath}
     u_i<u_j\quad\Longleftrightarrow\quad\exists\;k\;:\; i\in F_{k}
     \mbox{ and } j\not\in F_{k}\enspace .
   \end{displaymath}
   In particular,
   \begin{displaymath}
     u_i=u_j\quad\Longleftrightarrow\quad\exists\;k\;:\; i,j\in
     F_k\setminus F_{k-1}\enspace .
   \end{displaymath}
   The \emph{weight class} of a flag $\mathcal{F}$ is the set of all
   $u$ such that $\mathcal{F}(u)=\mathcal{F}$, which is a polyhedral cone in $\RR^n$.
   For example, the set of all
   vectors $u$ satisfying $u_3<u_1=u_4<u_2$ is the weight class in
   $\mathbb{R}^4$ corresponding to to the
   flag $\emptyset\subsetneq \{3\} \subsetneq \{1,3,4\} \subsetneq \{1,2,3,4\}$.

   A flag $\mathcal{F}$ is a \emph{flag of flats} of the matrix $A$ 
   (or of its associated matroid) if the linear span
   of the vectors $\{a_j \;|\; j\in F_i\}$ contains no $a_k$ with
   $k\notin F_i$. As before, the vectors $a_j$ denote the columns of $A$.
   
\begin{theorem}[{\cite[Theorem 1]{AK06}, \cite[Theorem 4.1]{FS05}}]\label{thm-Bergmanfan}
Let $X$ be a constant coefficient linear space, i.e., defined over a subfield with trivial valuation. The tropicalization of $X$, $\trop(X)$, equals the Bergman fan of the matroid
  of $X$, i.e., the union of all weight classes of flags of flats of
   the matroid of $X$.
\end{theorem}

The fan structure given by the flags of flats is the \emph{fine subdivision} of the Bergman fan. Generally, there exist fans with coarser structure and same support. By abuse of notation, we use the term Bergman fan for any fan supported on the tropicalization of our linear space, and do not necessarily fix a fan structure.


If $X$ is not a constant coefficient linear space, its tropicalization is not a fan, but a more general polyhedral complex which can be described in terms of the underlying valuated matroid. One can use the software \texttt{polymake} \cite{polymake} to compute tropicalizations of linear spaces with constant or non-constant coefficients, for further details see \cite{Rincon:2013} and \cite{HampeJoswigSchroeter:2018}. We used this to compute Example \ref{ex:valuated}, which has non-constant coefficients.

The interested reader will find more details about tropicalizations of linear spaces, their tropical moduli spaces, and their underlying (valuated) matroidal structure in \cite{Spe08, MS15, SchroeterThesis}.

\subsection{The tropicalization of the moduli space of codimension-1 spaces}
\label{subsec-tropicalization}

As the tropicalization of a constant coefficient linear subspace is the Bergman fan of its associated matroid (see Theorem~\ref{thm-Bergmanfan}), the following result describing the tropicalization of the locus of codimension-1 linear subspaces in a constant coefficient linear space $X$ is an immediate consequence of Theorem \ref{thm-main}. 
 
\begin{corollary}\label{cor-main}
Let $X \subset \mathbb{P}^{n-1}$ be a $d$-dimensional constant coefficient linear subspace, 
and $L(X) \subset \Gr(d,n) \subset \PP^{\binom{n}{d}-1}$ be the locus of $(d-1)$-dimensional linear subspaces contained in $X$. 
Then $\trop(L(X))$ is the Bergman fan of
\begin{enumerate}[label=(\roman*)]
	\item\label{item:tropcolsV}
		the matroid of dependencies among the columns of the matrix $V$, as defined in Equation \eqref{defV} (obtained from a matrix $W$ whose rowspace is $X$),
	\item\label{item:tropdualcolsU} 
		the dual matroid to the matroid of dependencies among the columns of the matrix $U$, as defined in Equation \eqref{defU}, and
	\item\label{item:troplines}
		the matroid of lines of the hyperplane arrangement $\mathcal A(X)$ induced by $X$ (see Definition \ref{def-matroidoflines}).
\end{enumerate}
\end{corollary}   
 
 \begin{remark}
 An analog of Theorem \ref{thm-main} also holds when $X$ has non-constant coefficients, see Remark \ref{rem-nonconstantcoefficients}. In this case, the tropicalization of $L(X)$ is encoded by the valuated matroid of the columns of $V$ or the dual valuated matroid of the columns of $U$, see e.g.\ Example \ref{ex:valuated}.
 \end{remark}
 
If $X_1$ and $X_2$ are two $(d+1)$-dimensional linear subspaces tropicalizing to the same Bergman fan $\mathcal{X}$, it is still possible that the tropicalization of the moduli spaces of $d$-dimensional subspaces $\trop(L(X_1))$ and $\trop(L(X_2))$ differ (see, for instance, Example \ref{ex-genericity} below, and Examples 3.3 and 3.4 in \cite{La18}). 
However, $\trop(L(X))$ is constant whenever $X$ is a subspace in the generic set (see Proposition \ref{prop-generic}). 

\begin{corollary}\label{cor-generic}
There is an open dense subset (in the Zariski topology) of the Grassmannian $\Gr_{\CC}(d+1,n)$ over the complex numbers $\CC$
such that for all $X$ in this subset (which are of constant coefficients) the Bergman fan $\trop(L(X))$ is constant. 
This subset is the set of generic subspaces in the sense of Proposition \ref{prop-generic}.
\end{corollary} 

\begin{proof}
This follows directly from Proposition \ref{prop-generic} and Corollary \ref{cor-main}, as the Bergman fan $\trop(L(X))$ only depends on the matroid of $L(X)$, which is constant for $X$ in the generic set. In particular, the subset occurring here is the same as in Proposition \ref{prop-generic}.
\end{proof} 
For an example of the generic subset of a Grassmannian, see Example \ref{ex-u36} above. In the next example, we relate our computation in Example \ref{ex-u36} of the space of generic planes for the Bergman fan of $U_{3,6}$ to Lamboglia's Fano schemes \cite{La18}.

\begin{example}\label{ex-genericity} Let $\mathcal{X}$ be the Bergman fan of $U_{3,6}$. In \cite{La18}, Lamboglia considers the space of all (realizable) tropical lines contained in $\mathcal{X}$, and which is a $3$-dimensional polyhedral complex that strictly contains the tropicalized spaces of lines that we study here. She also provides concrete examples to show the space $\trop(L(X))$ depends on the choice of the lift $X$ of $\mathcal{X}$: two different linear subspaces $X_1$ and $X_2$ tropicalizing to $\mathcal{X}$ are presented, and the tropicalizations of their spaces of lines $L(X_1)$ and $L(X_2)$ are computed.
The linear subspace $X_1$ is the rowspace of the matrix 
\begin{align*}
\begin{pmatrix}
0&-271&-92&0&-13&-54\\
0&-18&-7&-1&0&-4\\
-1&12293&4173&0&588&2450\\
\end{pmatrix}
\end{align*}
(see Example 3.3 in \cite{La18}). This subspace is generic in our sense: the relations given in Example \ref{ex-u36} are all nonzero (where we use the dehomogenized version involving $q_{123} = 1$). The linear subspace $X_2$ is the rowspace of the matrix
\begin{align*}
\begin{pmatrix}
1&3&0&1&5&7\\
0&0&1&3&-1&-1\\
1&4&-1&-3&0&0\\
\end{pmatrix}
\end{align*}
(see Example 3.4 in \cite{La18}). Even though all Pl\"ucker coordinates $q_I$ are still nonzero in this case, this subspace satisfies $q_{124}\cdot q_{356}-q_{123}\cdot q_{456}=0$.
This equation is one of our relations of Example \ref{ex-u36} when $q_{123}=1$.
It follows that $L(X_2)$ is non-generic, whereas $L(X_1)$ is. In particular, the Bergman fans of the matroids of $L(X_2)$ and $L(X_1)$ do not coincide. Examples 3.3 and 3.4 in \cite{La18} construct a tropical line $\Gamma$ in $\mathcal{X}$ which is not contained in the tropicalization of $L(X_1)$ but is contained in the tropicalization of $L(X_2)$. We will continue this example in Example \ref{ex-genericity2}.
\end{example}

The tropicalization of a generic $(d+1)$-dimensional subspace $X$ is the Bergman fan $\mathcal X$ of the uniform matroid $U_{d+1,n}$, and we use the notation $L^{\trop}(\mathcal{X})$ to denote the space $\trop(L(X))$ for such generic $X$. We call $L^{\trop}(\mathcal{X})$ the \emph{generic space of tropicalized $d$-dimensional subspaces for $\mathcal X$}. Making use of the results of Section \ref{sec-genericity}, it follows that $L^{\trop}(\mathcal{X})$ is the Bergman fan of a Dilworth truncation.
 
 \begin{corollary}\label{cor-DW}
If $X \in \Gr(d+1,n)$ is a generic linear subspace
(see Proposition~\ref{prop-generic}), then $\trop(L(X))=L^{\trop}(\mathcal{X})$ is the Bergman fan of $\widetilde \Dil_{n-d}(U_{n,n})$.
\end{corollary}
\begin{proof}
This follows directly from Theorem \ref{theorem-DW} and Corollary \ref{cor-main}.
\end{proof}

In the case that the linear subspace $X$ is not generic, Corollary \ref{cor-main} describes the tropicalization $\trop(L(X))$ as the Bergman fan of the matroid of lines of hyperplane arrangement $\mathcal A(X)$. 

 \begin{example}\label{ex-genericity2}
We continue Example \ref{ex-genericity}, and compute the matroids that correspond to the two tropicalized Fano schemes.
The linear subspace $X_1$ 
is generic, so we can deduce that the matroid of $L(X_1)$ is the relabeled $4$-th Dilworth truncation $\widetilde \Dil_4 (U_{6,6})$, by Corollary \ref{cor-DW}.
It is the matroid with $15$ elements that arises from intersecting six generic lines in~$\mathbb{P}^2$.

The linear subspace $X_2$ 
is not generic, so we cannot apply Corollary \ref{cor-DW} to compute the matroid of $L(X_2)$. 
Nevertheless, this matroid is the matroid of lines of the hyperplane arrangement $\mathcal A(X_2)$, following Theorem \ref{thm-main}. It is the matroid of the point configuration arising as the intersections of six lines in $\mathbb{P}^2$, but these lines are not generic --- concretely, the points $p_{12}$, $p_{34}$ and $p_{56}$ lie on a line. This is similar to the situation of Example~\ref{ex:matroid_of_lines} that we proceed to analyze in Example~\ref{ex:degenerate_planeR3}.
\end{example}

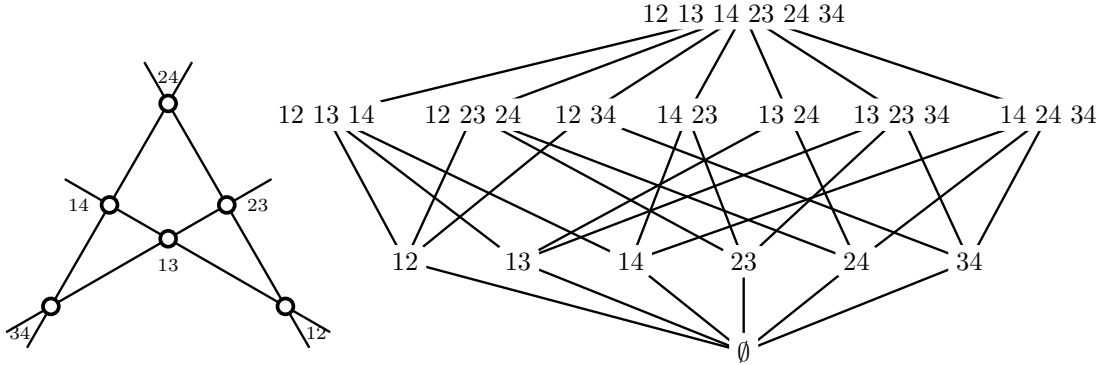
\begin{figure}
\begin{center}
\begin{tikzpicture}[scale = 1.5,
                    edge/.style={black, line width=0.9pt, line cap=round},
                    color = {black}]


  \tikzstyle{vertex}=[draw=black, fill=white, line width=1.4pt]
  \coordinate (trans) at (-.6,-.2);
  
  \coordinate (p13) at (trans); 
  \coordinate (p24) at ($(90:1.2)+(trans)$); 
  \coordinate (p12) at ($(330:1.2)+(trans)$); 
  \coordinate (p23) at ($.5*(90:1.2)+.5*(330:1.2)+(trans)$); 
  \coordinate (p34) at ($(210:1.2)+(trans)$); 
  \coordinate (p14) at ($.5*(p24)+.5*(p34)$); 

  \node[label=below right:{\tiny $12$}]  at (p12) {};
  \node[label=below:{\tiny $13$}]  at (p13) {};
  \node[label=left:{\tiny $14$}]  at (p14) {};
  \node[label=right:{\tiny $23$}]  at (p23) {};
  \node[label=above:{\tiny $24$}]  at (p24) {};
  \node[label=below left:{\tiny $34$}]  at (p34) {};

  \draw[edge] ($1.2*(p24)-.2*(p34)$) -- ($1.2*(p34)-.2*(p24)$);
  \draw[edge] ($1.2*(p24)-.2*(p12)$) -- ($1.2*(p12)-.2*(p24)$);
  \draw[edge] ($1.25*(p14)-.25*(p12)$) -- ($1.25*(p12)-.25*(p14)$);
  \draw[edge] ($1.25*(p23)-.25*(p34)$) -- ($1.25*(p34)-.25*(p23)$);

  \draw[vertex] (p12) circle (2pt);
  \draw[vertex] (p13) circle (2pt);
  \draw[vertex] (p14) circle (2pt);
  \draw[vertex] (p23) circle (2pt);
  \draw[vertex] (p24) circle (2pt);
  \draw[vertex] (p34) circle (2pt);


\tikzstyle{node}=[text=black, inner sep=3pt, rectangle, rounded corners=3pt,fill=white, draw=none]

\coordinate (H1) at (0.8,.9);
\coordinate (H2) at (2.1,.9);
\coordinate (H5) at (3.1,.9);
\coordinate (H7) at (4,.9);
\coordinate (H6) at (4.9,.9);
\coordinate (H3) at (5.9,.9);
\coordinate (H4) at (7.2,.9);
\coordinate (v12) at (1.5,-.4);
\coordinate (v13) at (2.5,-.4);
\coordinate (v14) at (3.5,-.4);
\coordinate (v23) at (4.5,-.4);
\coordinate (v24) at (5.5,-.4);
\coordinate (v34) at (6.5,-.4);
\coordinate (b) at (4.5,-1.2);
\coordinate (t) at (4.5,1.8);

\foreach \i/\k in {1/12,1/13,1/14,2/12,2/23,2/24,3/13,3/23,3/34,4/14,4/24,4/34,5/12,5/34,6/13,6/24,7/14,7/23} {
   \draw[edge] (H\i) -- (v\k);
   }

\foreach \i in {1,2,3,4,5,6,7} {
   \draw[edge] (H\i) -- (t);
  }
  
\foreach \k in {12,13,14,23,24,34} {
   \draw[edge] (v\k) -- (b);
  }

\node[node]  at (b) {\small $\emptyset$};

 \node[node] at (H1) {\small $12$ $13$ $14$};
 \node[node] at (H2) {\small $12$ $23$ $24$};
 \node[node] at (H3) {\small $13$ $23$ $34$};
 \node[node] at (H4) {\small $14$ $24$ $34$};
 
 \node[node] at (H5) {\small $12$ $34$};
 \node[node] at (H6) {\small $13$ $24$};
 \node[node] at (H7) {\small $14$ $23$};
  
\foreach \k in {12,13,14,23,24,34} {
   \node[node] at (v\k) {\small $\k$};
  }

\node[node] at (t) {\small $12$ $13$ $14$ $23$ $24$ $34$};

\end{tikzpicture}
\caption{The affine point configuration realizing the (relabeled) Dilworth truncation $\widetilde{\Dil}_{2}(U_{4,4})$ , which is the matroid of $L^{\trop}(\mathcal{X})$ for the Bergman fan $\mathcal X$ of the matroid $U_{3,4}$ on the left. And its lattice of flats on the right. }\label{fig-DW}
\end{center}
\end{figure}

As mentioned in the introduction, the next example shows that, in general, there is no identification of the plane $\mathcal{X}$ with its tropicalized dual plane $L^{\trop}(\mathcal{X})$, as their associated matroids might not agree.

\begin{example}\label{ex-standardplaneR3}
Let $\mathcal{X}$ be the Bergman fan of $U_{3,4}$, i.e., the standard tropical plane in $\mathbb{R}^3$ (see the left part of Figure \ref{fig-U34}).
By Theorem \ref{theorem-DW}, the matroid of the generic space of tropicalized lines $L^{\trop}(\mathcal{X})$ is equal to the Bergman fan of $\widetilde{\Dil}_{2}(U_{4,4})$. As discussed in Example \ref{ex-n-2DWgeom}, the underlying matroid $\widetilde{\Dil}_{2}(U_{4,4})$ is realized by six points labeled $12$, $13$, $14$, $23$, $24$, $34$ in $\mathbb{P}^2$, such that, e.g., $12$, $13$, $14$ lie on a line, etc. This matroid is illustrated on the left in Figure \ref{fig-DW}.

The Bergman fan of this matroid equals the generic space of tropicalized lines $L^{\trop}(\mathcal{X})$. In its fine subdivision, it has 13 rays corresponding to the 13 flats of $\widetilde \Dil_{2}(U_{4,4})$ --- 6 points, 4 lines going through 3 points each, and 3 lines going through 2 points. For the lattice of flats, see the right part of Figure \ref{fig-DW}. In this Bergman fan, there is a $2$-dimensional cone spanned by two rays if and only if the two rays correspond to a point and a line on such that the point is incident with the line. This includes the lines that only have two incident points, e.g.\ the line with the points $12$ and $34$.
The link of this Bergman fan is the Petersen graph, with three additional vertices corresponding to the three lines through only two points (see Figure \ref{fig-Bergman1}).
We obtain the coarse subdivision by dropping those three extra vertices. 

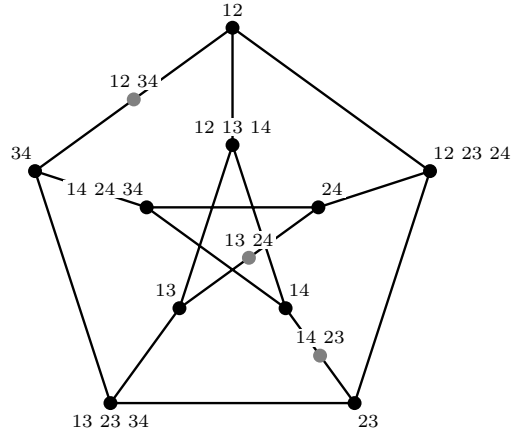
\begin{figure}[ht]
\begin{center}
\begin{tikzpicture}[scale = 1.2,
                    edge/.style={black, line width=0.9pt, line cap=round},
                    vertex/.style = {text=black, inner sep=2pt, white, draw=none}
                    ]

  \coordinate (p1) at (90:1); 
  \coordinate (p2) at (18:2.3); 
  \coordinate (p3) at (234:2.3); 
  \coordinate (p4) at (162:1); 

  \coordinate (p12) at (90:2.3); 
  \coordinate (p13) at (234:1); 
  \coordinate (p14) at (306:1); 
  \coordinate (p23) at (306:2.3); 
  \coordinate (p24) at (18:1); 
  \coordinate (p34) at (162:2.3);

  \draw[edge] (p12) -- (p34) -- (p3) -- (p23) -- (p2) -- cycle; 
  \draw[edge] (p1) -- (p13) -- (p24) -- (p4) -- (p14) -- cycle; 
  \draw[edge] (p1) -- (p12); 
  \draw[edge] (p2) -- (p24); 
  \draw[edge] (p3) -- (p13); 
  \draw[edge] (p4) -- (p34); 
  \draw[edge] (p14) -- (p23); 

  \fill (p1) circle [radius=2.2pt];
  \fill (p2) circle [radius=2.2pt];
  \fill (p3) circle [radius=2.2pt];
  \fill (p4) circle [radius=2.2pt];

  \fill (p12) circle [radius=2.2pt];
  \fill (p13) circle [radius=2.2pt];
  \fill (p14) circle [radius=2.2pt];
  \fill (p23) circle [radius=2.2pt];
  \fill (p24) circle [radius=2.2pt];
  \fill (p34) circle [radius=2.2pt];
  \fill[gray] ($.5*(p12)+.5*(p34)$) circle [radius=2.2pt];
  \fill[gray] ($.5*(p13)+.5*(p24)$) circle [radius=2.2pt];
  \fill[gray] ($.5*(p14)+.5*(p23)$) circle [radius=2.2pt];
  
  \node[above, fill=white, inner sep=1pt] at ($(p1)+(0,0.1)$) {\tiny $12$ $13$ $14$};
  \node[above right, fill=white, inner sep=1pt] at ($(p2)+(0,0.1)$) {\tiny $12$ $23$ $24$};
  \node[below, fill=white, inner sep=1pt] at ($(p3)-(0,0.1)$) {\tiny $13$ $23$ $34$};
  \node[above left, fill=white, inner sep=1pt] at ($(p4)+(0,0.1)$) {\tiny $14$ $24$ $34$};

  \node[above, fill=white, inner sep=1pt] at ($(p12)+(0,0.1)$){\tiny $12$};
  \node[above left, fill=white, inner sep=1pt] at ($(p13)+(0,0.1)$){\tiny $13$}; 
  \node[above right, fill=white, inner sep=1pt] at ($(p14)+(0,0.1)$){\tiny $14$};
  \node[below right, fill=white, inner sep=1pt] at ($(p23)-(0,0.1)$){\tiny $23$};    
  \node[above right, fill=white, inner sep=1pt] at ($(p24)+(0,0.1)$){\tiny $24$};        
  \node[above left, fill=white, inner sep=1pt] at ($(p34)+(0,0.1)$){\tiny $34$};    
  
  \node[above, fill=white, inner sep=1pt] at ($.5*(p12)+.5*(p34)+(0,0.1)$) {\tiny $12$ $34$};
  \node[above, fill=white, inner sep=1pt] at ($.5*(p13)+.5*(p24)+(0,0.1)$) {\tiny $13$ $24$};
  \node[above, fill=white, inner sep=1pt] at ($.5*(p14)+.5*(p23)+(0,0.1)$) {\tiny $14$ $23$};

\end{tikzpicture}
\caption{The link of the Bergman fan of $\widetilde\Dil_{2}(U_{4,4})$ with the fine subdivision agrees with the link of the moduli space $L^{\trop}(\mathcal{X})$, and is the Petersen graph with $3$ extra vertices (gray).}\label{fig-Bergman1}
\end{center}
\end{figure}

We now describe exactly how the Bergman fan $L^{\trop}(\mathcal{X})$ parametrizes lines inside $\mathcal X$. 

The fan $\mathcal{X}$ in its coarse subdivision has four rays, say $1$, $2$, $3$, and $4$. To obtain its fine subdivision we subdivide it by adding the six rays $12$, $13$, $14$, $23$, $24$, $34$, i.e., we subdivide each of the six $2$-dimensional cones, obtaining a total of $12$ maximal cones. Each such $2$-dimensional cone of $\mathcal{X}$ is spanned by two rays, of the form $i$ and $ij$. The fan $\mathcal{X}$ in its fine subdivision 
is depicted in the left part of Figure \ref{fig-U34}.

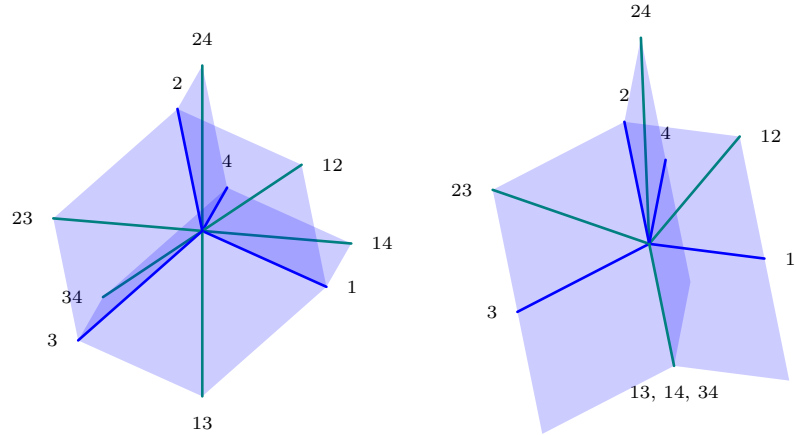
\begin{figure}[ht]
\begin{center}
\begin{tikzpicture}[
	x  = {(-0.9cm,0.076cm)},
        y  = {(0.0cm,-1cm)},
        z  = {(.6cm,.4cm)},
  edge/.style={line width=1pt, line cap=round, teal},
  edge2/.style={line width=1pt, line cap=round, blue},
  face/.style={draw=none, opacity=0.2, fill=blue},
  scale=1.1
  ]

  \coordinate (p0) at (0,0,0); 
  \coordinate (p1) at (-1,1,1); 
  \coordinate (p2) at (1,-1,1); 
  \coordinate (p3) at (1,1,-1); 
  \coordinate (p4) at (-1,-1,-1); 

  \coordinate (p12) at ($(p1)+(p2)$);
  \coordinate (p13) at ($(p1)+(p3)$);
  \coordinate (p14) at ($(p1)+(p4)$);
  \coordinate (p23) at ($(p2)+(p3)$);
  \coordinate (p24) at ($(p2)+(p4)$);
  \coordinate (p34) at ($(p3)+(p4)$);

  \node[label = right:{\tiny $1$}]  at (p1) {};
  \node[label = above:{\tiny $2$}]  at (p2) {};
  \node[label = left:{\tiny $3$}]  at (p3) {};
  \node[label = above:{\tiny $4$}]  at (p4) {};

  \node[label = right:{\tiny $12$}]  at (p12) {};
  \node[label = below:{\tiny $13$}]  at (p13) {};
  \node[label = right:{\tiny $14$}]  at (p14) {};
  \node[label = left:{\tiny $23$}]  at (p23) {};
  \node[label = above:{\tiny $24$}]  at (p24) {};
  \node[label = left:{\tiny $34$}]  at (p34) {};

  \draw[face] (p0) -- (p2) -- (p24) -- (p4) -- cycle;
  \draw[edge] (p0) -- (p24); 
  \draw[face] (p0) -- (p1) -- (p14) -- (p4) -- cycle;
  \draw[edge] (p0) -- (p14); 
  \draw[face] (p0) -- (p3) -- (p34) -- (p4) -- cycle;
  \draw[edge] (p0) -- (p34); 
  \draw[edge2] (p0) -- (p4); 

  \draw[face] (p0) -- (p1) -- (p12) -- (p2) -- cycle; 
  \draw[edge] (p0) -- (p12);
  \draw[edge2] (p0) -- (p1);

  \draw[face] (p0) -- (p1) -- (p13) -- (p3) -- cycle; 
  \draw[edge] (p0) -- (p13);

  \draw[face] (p0) -- (p2) -- (p23) -- (p3) -- cycle; 
  \draw[edge] (p0) -- (p23);
  \draw[edge2] (p0) -- (p2); 
  \draw[edge2] (p0) -- (p3);

  \coordinate (trans) at (-6,-.3,0);
  \coordinate (q0) at ($(0,0,0)+(trans)$); 
  \coordinate (q2) at ($(1,-1,1)+(trans)$); 

  \coordinate (q1) at ($(-.66,.66,1.33)+(trans)$);
  \coordinate (q3) at ($(1.33,.66,-.66)+(trans)$);
  \coordinate (q4) at ($(-.66,-1.33,-.66)+(trans)$);
    
  \coordinate (q12) at ($(q1)+(q2)-(trans)$);
  \coordinate (q23) at ($(q3)+(q2)-(trans)$);
  \coordinate (q24) at ($(q4)+(q2)-(trans)$);
    
  \coordinate (qX) at ($(-1,1,-1)+(trans)$);

  \node[label = right:{\tiny $12$}]  at (q12) {};
  \node[label = left:{\tiny $23$}]  at (q23) {};
  \node[label = above:{\tiny $24$}]  at (q24) {};
  \node[label = below:{\tiny $13$, $14$, $34$}]  at (qX) {};
  \node[label=above:{\tiny $2$}]  at (q2) {};
  
  \node[label= above:{\tiny $4$}]  at ($(q24)+(qX)-(trans)$) {};
  \node[label= right:{\tiny $1$}]  at ($(q12)+(qX)-(trans)$) {};
  \node[label= left:{\tiny $3$}]  at ($(q23)+(qX)-(trans)$) {};
  
  \draw[face] (q24) -- (q2) -- (qX) -- ($(q4)+(qX)-(trans)$) -- cycle;
  \draw[edge] (q0) -- (q24);
  \draw[face] (q12) -- (q2) -- (qX) -- ($(q1)+(qX)-(trans)$) -- cycle;
  \draw[edge] (q0) -- (q12);

  \draw[edge2] (q0) -- (q1); 
  \draw[edge2] (q0) -- (q3); 
  \draw[edge2] (q0) -- (q4); 

  \draw[face] (q23) -- (q2) -- (qX) -- ($(q3)+(qX)-(trans)$) -- cycle;
  \draw[edge] (q0) -- (q23);
  \draw[edge2] (q0) -- (q2); 
  \draw[edge] (q0) -- (qX);

\end{tikzpicture}
\caption{The Bergman fans $\mathcal{X}$ of Example~\ref{ex-standardplaneR3}  (left) and $\mathcal{X}'$ of Example~\ref{ex:degenerate_planeR3} (right) in their fine subdivisions.}\label{fig-U34}
\end{center}
\end{figure}

We computed the positions of the vertices of a tropical line in $\mathbb{R}^3$ from its Pl\"ucker coordinates (using, e.g.\ \cite[Example 4.3.19]{MS15}). 

For a ray of the form $\{ij\}$ of the Bergman fan $L^{\trop}(\mathcal{X})$, the Pl\"ucker vector corresponding to its primitive generating vector has a $-1$ in coordinate $ij$ and $0$ otherwise. The corresponding tropical line in $\mathcal X$ has thus vertices at the origin $(0,0,0,0)$ and $-(e_i+e_j)$. 
This ray thus parametrizes tropical lines with one vertex at the vertex of $\mathcal{X}$ and one vertex in the ray $ij$. 
A ray of the form $\{ij, ik, il\}$ yields a Pl\"ucker vector with coordinates $-1$ in those three entries and $0$ otherwise. This corresponds to a $4$-valent tropical line with vertex at $-e_i$.
This ray thus parametrizes tropical lines in $\mathcal X$ with a $4$-valent vertex at the ray $i$ of $\mathcal{X}$. We therefore choose to relabel it using the letter $i$ (see the left part of Figure \ref{fig-Bergman2}).

Analogously, for a ray of the form $\{ij, kl\}$, we obtain the tropical line with vertices at $-(e_i+e_j)$ and $-(e_k+e_l)$. These are the rays we drop in the coarse subdivision.
The ray parametrizes tropical lines whose bounded edge passes through the vertex of $\mathcal{X}$, such that it is divided into two equal parts by the vertex of $\mathcal{X}$. 
The tropical lines in the two adjacent $2$-dimensional cones differ just by which part of the bounded edge is longer. Since this is not a combinatorial difference, from the point of view of combinatorial types of lines in $\mathcal{X}$ it hence makes sense to drop these $3$ extra rays and use the coarse subdivision for the Bergman fan $L^{\trop}(\mathcal{X})$.

Having discussed the lines corresponding to the rays of the Bergman fan $L^{\trop}(\mathcal{X})$, we can now turn to its $2$-dimensional cones. The $2$-dimensional cones spanned by rays of the form $i$ and $ij$ parametrize tropical lines with one vertex on the ray $i$ and one vertex in the cone spanned by $i$ and $ij$ of $\mathcal{X}$. There are 12 such $2$-dimensional cones, corresponding to the twelve $2$-dimensional cones of $\mathcal{X}$ that contain one vertex of the line in their interior.
The $2$-dimensional cones spanned by rays of the form $ij$ and $kl$ parametrize tropical lines whose bounded edge passes through the vertex of $\mathcal{X}$. There are $3$ such $2$-dimensional cones, corresponding to the three directions that a bounded edge of a tropical line may have.

We can see how the fifteen $2$-dimensional cones of the fan $L^{\trop}(\mathcal{X})$ (in its coarse subdivision) correspond to combinatorial types of tropical lines in $\mathcal{X}$. The fan, using the coarsening and relabeling from above (i.e., dropping the three extra rays and denoting a ray $ij$, $ik$, $il$ by $i$ instead) is depicted in the left part of Figure \ref{fig-Bergman2}. The combinatorial types of the lines in $\mathcal X$ parametrized by the rays of $L^{\trop}(\mathcal{X})$ are depicted in Figure \ref{fig-combtypes} (up to symmetry). The combinatorial types of the lines in $\mathcal X$ parametrized by the $2$-dimensional cones of $L^{\trop}(\mathcal{X})$ are depicted in Figure \ref{fig-combtypes2} (again up to symmetry).

Note that in this example all tropical lines contained in $\mathcal X$ are actually obtained as the tropicalization of a classical line in $X$, for any generic lift $X$ of $\mathcal X$. This is not the case in general, as we discussed, for instance, in Example \ref{ex-genericity}.

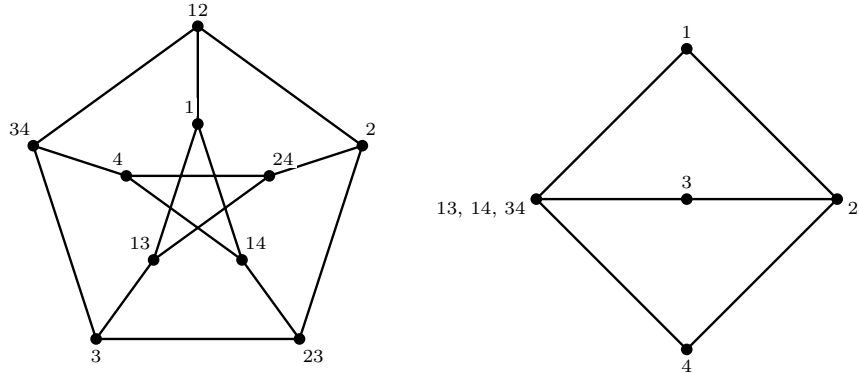
\begin{figure}[ht]
\begin{center}
\begin{tikzpicture}[scale = 1,
                    edge/.style={black, line width=0.9pt, line cap=round},
                    vertex/.style = {text=black, inner sep=2pt, white, draw=none}
                    ]

  \coordinate (p1) at (90:1); 
  \coordinate (p2) at (18:2.3); 
  \coordinate (p3) at (234:2.3); 
  \coordinate (p4) at (162:1); 

  \coordinate (p12) at (90:2.3); 
  \coordinate (p13) at (234:1); 
  \coordinate (p14) at (306:1); 
  \coordinate (p23) at (306:2.3); 
  \coordinate (p24) at (18:1); 
  \coordinate (p34) at (162:2.3);

  \draw[edge] (p12) -- (p34) -- (p3) -- (p23) -- (p2) -- cycle; 
  \draw[edge] (p1) -- (p13) -- (p24) -- (p4) -- (p14) -- cycle; 
  \draw[edge] (p1) -- (p12); 
  \draw[edge] (p2) -- (p24); 
  \draw[edge] (p3) -- (p13); 
  \draw[edge] (p4) -- (p34); 
  \draw[edge] (p14) -- (p23); 

  \fill (p1) circle [radius=2.2pt];
  \fill (p2) circle [radius=2.2pt];
  \fill (p3) circle [radius=2.2pt];
  \fill (p4) circle [radius=2.2pt];

  \fill (p12) circle [radius=2.2pt];
  \fill (p13) circle [radius=2.2pt];
  \fill (p14) circle [radius=2.2pt];
  \fill (p23) circle [radius=2.2pt];
  \fill (p24) circle [radius=2.2pt];
  \fill (p34) circle [radius=2.2pt];

  \node[above left, fill=white, inner sep=1pt] at ($(p1)+(0,0.1)$) {\tiny $1$};
  \node[above right, fill=white, inner sep=1pt] at ($(p2)+(0,0.1)$) {\tiny $2$};
  \node[below, fill=white, inner sep=1pt] at ($(p3)-(0,0.1)$) {\tiny $3$};
  \node[above left, fill=white, inner sep=1pt] at ($(p4)+(0,0.1)$) {\tiny $4$};

  \node[above, fill=white, inner sep=1pt] at ($(p12)+(0,0.1)$){\tiny $12$};
  \node[above left, fill=white, inner sep=1pt] at ($(p13)+(0,0.1)$){\tiny $13$}; 
  \node[above right, fill=white, inner sep=1pt] at ($(p14)+(0,0.1)$){\tiny $14$};
  \node[below right, fill=white, inner sep=1pt] at ($(p23)-(0,0.1)$){\tiny $23$};    
  \node[above right, fill=white, inner sep=1pt] at ($(p24)+(0,0.1)$){\tiny $24$};        
  \node[above left, fill=white, inner sep=1pt] at ($(p34)+(0,0.1)$){\tiny $34$};


  \coordinate (trans) at (6.5,0);
  
  \coordinate (q2) at ($(0:2)+(trans)$); 
  \coordinate (q12) at ($(90:2)+(trans)$); 
  \coordinate (q24) at ($(270:2)+(trans)$); 
  \coordinate (q23) at ($(0:0)+(trans)$); 
  \coordinate (qX) at ($(180:2)+(trans)$);

  \node[below right, fill=white, inner sep=1pt] at ($(q2)+(0.1,0)$) {\tiny $2$};
  \node[below left, fill=white, inner sep=1pt] at ($(qX)-(0.1,0)$) {\tiny $13$, $14$, $34$};
  \node[above, fill=white, inner sep=1pt] at ($(q12)+(0,0.1)$) {\tiny $1$};
  \node[above, fill=white, inner sep=1pt] at ($(q23)+(0,0.1)$) {\tiny $3$};
  \node[below, fill=white, inner sep=1pt] at ($(q24)-(0,0.1)$) {\tiny $4$};
  
  \draw[edge] (q12) -- (q2) -- (q23) -- (qX) -- (q24) -- (q2); 
  \draw[edge] (q12) -- (qX); 

  \fill (q2) circle [radius=2.2pt];
  \fill (q12) circle [radius=2.2pt];
  \fill (q23) circle [radius=2.2pt];
  \fill (q24) circle [radius=2.2pt];
  \fill (qX) circle [radius=2.2pt];

\end{tikzpicture}
\caption{The link of the Bergman fan $L^{\trop}(\mathcal{X})$, on the left, in its coarse subdivision with relabeled rays (Example \ref{ex-standardplaneR3}). 
The link of the Bergman fan $\trop(L(X))$, on the right, in its coarse subdivision with relabeled rays (Example \ref{ex:degenerate_planeR3}). 
}\label{fig-Bergman2}
\end{center}
\end{figure}

\begin{figure}[ht]
\begin{center}
\begin{tikzpicture}[
	x  = {(-0.9cm,0.076cm)},
        y  = {(0.0cm,-1cm)},
        z  = {(.6cm,.4cm)},
  edge/.style={line width=.8pt, line cap=round, teal},
  edge2/.style={line width=.8pt, line cap=round, blue},
  face/.style={draw=none, opacity=0.2, fill=blue},
  vertex/.style={draw = red, fill = red, line width=1.5pt},
  edge3/.style={line width=1.3pt, line cap=round, red, dotted},
  scale=1.2
  ]

  \coordinate (p0) at (0,0,0); 
  \coordinate (p1) at (-1,1,1); 
  \coordinate (p2) at (1,-1,1); 
  \coordinate (p3) at (1,1,-1); 
  \coordinate (p4) at (-1,-1,-1); 

  \coordinate (p12) at ($(p1)+(p2)$);
  \coordinate (p13) at ($(p1)+(p3)$);
  \coordinate (p14) at ($(p1)+(p4)$);
  \coordinate (p23) at ($(p2)+(p3)$);
  \coordinate (p24) at ($(p2)+(p4)$);
  \coordinate (p34) at ($(p3)+(p4)$);

  \node[label = right:{\tiny $1$}]  at (p1) {};
  \node[label = above:{\tiny $2$}]  at (p2) {};
  \node[label = left:{\tiny $3$}]  at (p3) {};
  \node[label = above:{\tiny $4$}]  at (p4) {};

  \node[label = right:{\tiny $12$}]  at (p12) {};
  \node[label = below:{\tiny $13$}]  at (p13) {};
  \node[label = right:{\tiny $14$}]  at (p14) {};
  \node[label = left:{\tiny $23$}]  at (p23) {};
  \node[label = above:{\tiny $24$}]  at (p24) {};
  \node[label = left:{\tiny $34$}]  at (p34) {};

  \draw[face] (p0) -- (p2) -- (p24) -- (p4) -- cycle;
  \draw[edge] (p0) -- (p24); 
  \draw[face] (p0) -- (p1) -- (p14) -- (p4) -- cycle;
  \draw[edge] (p0) -- (p14); 
  \draw[face] (p0) -- (p3) -- (p34) -- (p4) -- cycle;
  \draw[edge] (p0) -- (p34); 
  \draw[edge2] (p0) -- (p4); 

  \draw[edge3] ($0.4*(p1)$) -- ($(p4)+.4*(p1)$);

  \draw[face] (p0) -- (p1) -- (p12) -- (p2) -- cycle; 
  \draw[edge] (p0) -- (p12);
  \draw[edge2] (p0) -- (p1);

  \draw[face] (p0) -- (p1) -- (p13) -- (p3) -- cycle; 
  \draw[edge] (p0) -- (p13);

  \draw[face] (p0) -- (p2) -- (p23) -- (p3) -- cycle; 
  \draw[edge] (p0) -- (p23);
  \draw[edge2] (p0) -- (p2); 
  \draw[edge2] (p0) -- (p3);

  \draw[edge3] ($0.4*(p1)$) -- (p1); 
  \draw[edge3] ($0.4*(p1)$) -- ($(p3)+.4*(p1)$); 
  \draw[edge3] ($0.4*(p1)$) -- ($(p2)+.4*(p1)$);
  \draw[vertex] ($0.4*(p1)$) circle (.6pt);

  \coordinate (trans) at (-6,-.4,0);
  \coordinate (q0) at ($(0,0,0)+(trans)$); 
  \coordinate (q1) at ($(-1,1,1)+(trans)$); 
  \coordinate (q2) at ($(1,-1,1)+(trans)$); 
  \coordinate (q3) at ($(1,1,-1)+(trans)$); 
  \coordinate (q4) at ($(-1,-1,-1)+(trans)$); 

  \coordinate (q12) at ($(q1)+(q2)-(trans)$);
  \coordinate (q13) at ($(q1)+(q3)-(trans)$);
  \coordinate (q14) at ($(q1)+(q4)-(trans)$);
  \coordinate (q23) at ($(q2)+(q3)-(trans)$);
  \coordinate (q24) at ($(q2)+(q4)-(trans)$);
  \coordinate (q34) at ($(q3)+(q4)-(trans)$);

  \node[label = right:{\tiny $1$}]  at (q1) {};
  \node[label = above:{\tiny $2$}]  at (q2) {};
  \node[label = left:{\tiny $3$}]  at (q3) {};
  \node[label = above:{\tiny $4$}]  at (q4) {};

  \node[label = right:{\tiny $12$}]  at (q12) {};
  \node[label = below:{\tiny $13$}]  at (q13) {};
  \node[label = right:{\tiny $14$}]  at (q14) {};
  \node[label = left:{\tiny $23$}]  at (q23) {};
  \node[label = above:{\tiny $24$}]  at (q24) {};
  \node[label = left:{\tiny $34$}]  at (q34) {};

  \draw[face] (q0) -- (q2) -- (q24) -- (q4) -- cycle;
  \draw[edge] (q0) -- (q24); 
  \draw[face] (q0) -- (q1) -- (q14) -- (q4) -- cycle;
  \draw[edge] (q0) -- (q14); 
  \draw[face] (q0) -- (q3) -- (q34) -- (q4) -- cycle;
  \draw[edge] (q0) -- (q34); 
  \draw[edge2] (q0) -- (q4); 

  \draw[edge3] (q0) -- (q4);

  \draw[face] (q0) -- (q1) -- (q12) -- (q2) -- cycle; 
  \draw[edge] (q0) -- (q12);
  \draw[edge2] (q0) -- (q1);

  \draw[face] (q0) -- (q1) -- (q13) -- (q3) -- cycle; 
  \draw[edge] (q0) -- (q13);

  \draw[face] (q0) -- (q2) -- (q23) -- (q3) -- cycle; 
  \draw[edge] (q0) -- (q23);
  \draw[edge2] (q0) -- (q2); 
  \draw[edge2] (q0) -- (q3); 

  \draw[edge3] (q0) -- (q2);
  \draw[edge3] ($0.4*(q13)+0.6*(trans)$) -- ($0.4*(q13)+0.6*(q1)$); 
  \draw[edge3] ($0.4*(q13)+0.6*(trans)$) -- ($0.4*(q13)+0.6*(q3)$);
  \draw[edge3] (q0) -- ($0.4*(q13)+0.6*(trans)$); 
  \draw[vertex] ($0.4*(q13)+0.6*(trans)$) circle (.6pt);
  \draw[vertex] (q0) circle (.6pt);
  
\end{tikzpicture}
\caption{The combinatorial type of the tropical lines parametrized by points on the rays labeled by $1$ and $13$ respectively in the moduli space $L^{\trop}(\mathcal{X})$ in its coarse subdivision (see the left part of Figure \ref{fig-Bergman2} and Example~\ref{ex-standardplaneR3}).
}\label{fig-combtypes}
\end{center}
\end{figure}
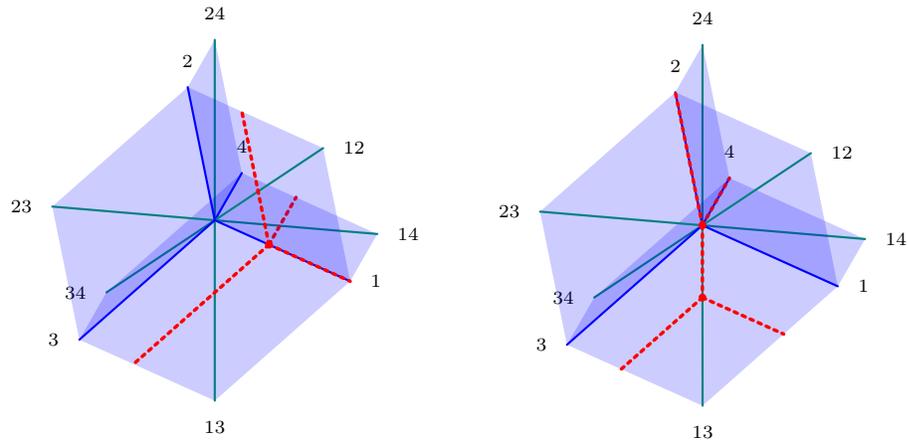
\begin{figure}
\begin{center}
\begin{tikzpicture}[
	x  = {(-0.9cm,0.076cm)},
        y  = {(0.0cm,-1cm)},
        z  = {(.6cm,.4cm)},
  edge/.style={line width=.8pt, line cap=round, teal},
  edge2/.style={line width=.8pt, line cap=round, blue},
  face/.style={draw=none, opacity=0.2, fill=blue},
  vertex/.style={draw = red, fill = red, line width=1.5pt},
  edge3/.style={line width=1.3pt, line cap=round, red, dotted},
  scale=1.2
  ]

  \coordinate (p0) at (0,0,0); 
  \coordinate (p1) at (-1,1,1); 
  \coordinate (p2) at (1,-1,1); 
  \coordinate (p3) at (1,1,-1); 
  \coordinate (p4) at (-1,-1,-1); 

  \coordinate (p12) at ($(p1)+(p2)$);
  \coordinate (p13) at ($(p1)+(p3)$);
  \coordinate (p14) at ($(p1)+(p4)$);
  \coordinate (p23) at ($(p2)+(p3)$);
  \coordinate (p24) at ($(p2)+(p4)$);
  \coordinate (p34) at ($(p3)+(p4)$);

  \node[label = right:{\tiny $1$}]  at (p1) {};
  \node[label = above:{\tiny $2$}]  at (p2) {};
  \node[label = left:{\tiny $3$}]  at (p3) {};
  \node[label = above:{\tiny $4$}]  at (p4) {};

  \node[label = right:{\tiny $12$}]  at (p12) {};
  \node[label = below:{\tiny $13$}]  at (p13) {};
  \node[label = right:{\tiny $14$}]  at (p14) {};
  \node[label = left:{\tiny $23$}]  at (p23) {};
  \node[label = above:{\tiny $24$}]  at (p24) {};
  \node[label = left:{\tiny $34$}]  at (p34) {};

  \draw[face] (p0) -- (p2) -- (p24) -- (p4) -- cycle;
  \draw[edge] (p0) -- (p24); 
  \draw[face] (p0) -- (p1) -- (p14) -- (p4) -- cycle;
  \draw[edge] (p0) -- (p14); 
  \draw[face] (p0) -- (p3) -- (p34) -- (p4) -- cycle;
  \draw[edge] (p0) -- (p34); 
  \draw[edge2] (p0) -- (p4); 

  \draw[edge3] ($0.4*(p1)$) -- ($(p4)+.4*(p1)$);

  \draw[face] (p0) -- (p1) -- (p12) -- (p2) -- cycle; 
  \draw[edge] (p0) -- (p12);
  \draw[edge2] (p0) -- (p1);

  \draw[face] (p0) -- (p1) -- (p13) -- (p3) -- cycle; 
  \draw[edge] (p0) -- (p13);

  \draw[face] (p0) -- (p2) -- (p23) -- (p3) -- cycle; 
  \draw[edge] (p0) -- (p23);
  \draw[edge2] (p0) -- (p2); 
  \draw[edge2] (p0) -- (p3); 

  \draw[edge3] ($0.7*(p1)+.3*(p3)$) -- ($(p1)+0.3*(p3)$); 
  \draw[edge3] ($0.7*(p1)+(p3)$) -- ($0.3*(p3)+.7*(p1)$); 
  \draw[edge3] ($0.4*(p1)$) -- ($(p2)+.4*(p1)$);
  \draw[edge3] ($0.7*(p1)+0.3*(p3)$) -- ($0.4*(p1)$);
  \draw[vertex] ($0.7*(p1)+0.3*(p3)$) circle (.6pt);
  \draw[vertex] ($0.4*(p1)$) circle (.6pt);

  \coordinate (trans) at (-6,-.4,0);
  \coordinate (q0) at ($(0,0,0)+(trans)$); 
  \coordinate (q1) at ($(-1,1,1)+(trans)$); 
  \coordinate (q2) at ($(1,-1,1)+(trans)$); 
  \coordinate (q3) at ($(1,1,-1)+(trans)$); 
  \coordinate (q4) at ($(-1,-1,-1)+(trans)$); 

  \coordinate (q12) at ($(q1)+(q2)-(trans)$);
  \coordinate (q13) at ($(q1)+(q3)-(trans)$);
  \coordinate (q14) at ($(q1)+(q4)-(trans)$);
  \coordinate (q23) at ($(q2)+(q3)-(trans)$);
  \coordinate (q24) at ($(q2)+(q4)-(trans)$);
  \coordinate (q34) at ($(q3)+(q4)-(trans)$);

  \node[label = right:{\tiny $1$}]  at (q1) {};
  \node[label = above:{\tiny $2$}]  at (q2) {};
  \node[label = left:{\tiny $3$}]  at (q3) {};
  \node[label = above:{\tiny $4$}]  at (q4) {};

  \node[label = right:{\tiny $12$}]  at (q12) {};
  \node[label = below:{\tiny $13$}]  at (q13) {};
  \node[label = right:{\tiny $14$}]  at (q14) {};
  \node[label = left:{\tiny $23$}]  at (q23) {};
  \node[label = above:{\tiny $24$}]  at (q24) {};
  \node[label = left:{\tiny $34$}]  at (q34) {};

  \draw[face] (q0) -- (q2) -- (q24) -- (q4) -- cycle;
  \draw[edge] (q0) -- (q24); 
  \draw[face] (q0) -- (q1) -- (q14) -- (q4) -- cycle;
  \draw[edge] (q0) -- (q14); 
  \draw[face] (q0) -- (q3) -- (q34) -- (q4) -- cycle;
  \draw[edge] (q0) -- (q34); 
  \draw[edge2] (q0) -- (q4); 

  \draw[edge3] ($0.7*(q24)+0.3*(trans)$) -- ($0.7*(q24)+0.3*(q2)$); 
  \draw[edge3] ($0.7*(q24)+0.3*(trans)$) -- ($0.7*(q24)+0.3*(q4)$);
  \draw[edge3] ($0.7*(q24)+0.3*(trans)$) -- (q0);
  \draw[vertex] ($0.7*(q24)+0.3*(trans)$) circle (.6pt);

  \draw[face] (q0) -- (q1) -- (q12) -- (q2) -- cycle; 
  \draw[edge] (q0) -- (q12);
  \draw[edge2] (q0) -- (q1);

  \draw[face] (q0) -- (q1) -- (q13) -- (q3) -- cycle; 
  \draw[edge] (q0) -- (q13);

  \draw[face] (q0) -- (q2) -- (q23) -- (q3) -- cycle; 
  \draw[edge] (q0) -- (q23);
  \draw[edge2] (q0) -- (q2); 
  \draw[edge2] (q0) -- (q3); 

  \draw[edge3] ($0.4*(q13)+0.6*(trans)$) -- ($0.4*(q13)+0.6*(q1)$); 
  \draw[edge3] ($0.4*(q13)+0.6*(trans)$) -- ($0.4*(q13)+0.6*(q3)$);
  \draw[edge3] (q0) -- ($0.4*(q13)+0.6*(trans)$); 
  \draw[vertex] ($0.4*(q13)+0.6*(trans)$) circle (.6pt);
  
\end{tikzpicture}
\caption{The combinatorial type of tropical lines parametrized by the $2$-dimensional cones spanned by rays $1$, $12$ and $12$, $34$ respectively in the moduli space  $L^{\trop}(\mathcal{X})$ in its coarse subdivision (see the left part of Figure \ref{fig-Bergman2} and Example \ref{ex-standardplaneR3}).}\label{fig-combtypes2}
\end{center}
\end{figure}
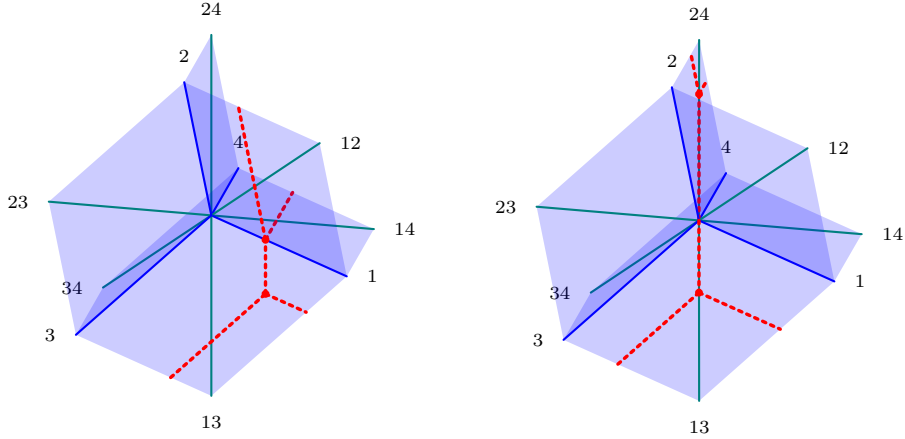
\end{example}

The next example uses Theorem~\ref{thm-main} to compute the matroid of the moduli space $L(X)$ in the case that $X$ is not a generic subspace. This example is a deformation of Example~\ref{ex-standardplaneR3}.

\begin{example}\label{ex:degenerate_planeR3} Let $\mathcal{X}'$ be the Bergman fan of the matroid $U_{2,3}\oplus U_{1,1}$, i.e., a degenerate plane equal to the standard tropical line in $\mathbb{R}^2$ with a $1$-dimensional lineality space attached (see the right part of Figure \ref{fig-U34}).
The matroid  $U_{2,3}\oplus U_{1,1}$ is 
(generically) realized by the matrix $W$ of Example \ref{ex:matroid_of_lines} (up to reordering the elements of the ground set). The Bergman fan $\mathcal{X}'$ is equal to the tropicalization of the rowspace $X$ of the matrix $W$. We computed the matroid of $L(X)$ in Example \ref{ex:matroid_of_lines}, which is illustrated in Figure \ref{fig:matroid_of_lines}. The link of its Bergman fan is depicted in the right part of Figure \ref{fig-Bergman2}.
It is worth mentioning that the underlying matroid of $L(X)$ is not simple and not a Dilworth truncation. 
 However, it is constant among all realizations $X$ of $U_{2,3}\oplus U_{1,1}$.
\end{example}

Our last example in this section studies the tropicalization of the moduli space $L(X)$ for a non-trivially valued subspace $X$.
 
\begin{example}\label{ex:valuated}
 Let $n=5$ and $d=2$. Consider the tropical linear space $\mathcal{X}''$ whose tropical Pl\"ucker coordinates are
 $\overline{q}_{125} = \overline{q}_{135} = 1$, $\overline{q}_{145}=3$, and the other seven coordinates are $0$.
 A projection of $\mathcal{X}''$ onto $\RR^3$ is visualized in Figure~\ref{fig:valuated}, although it does not reflect all its geometric aspects. 
 
 Suppose $X$ is a subspace that tropicalizes to $\mathcal X''$. In Example~\ref{ex-u36} we computed all the Pl\"ucker coordinates of the locus $L(X)$ in terms of the Pl\"ucker coordinates $q_I$ of $X$ (when $n\leq 6$), and gave a complete list of the ones that are not equal to monomials on the $q_I$. Since all terms appearing in this list involve a subset containing the element $6$, we conclude that, for $n\leq 5$, the Pl\"ucker coordinates of $L(X)$ are all monomials on the Pl\"ucker coordinates of $X$. In particular, this implies that the moduli space $\trop(L(X))$ is independent of the choice of lift $X$ of $\mathcal{X}''$. A concrete instance of a subspace $X$ that tropicalizes to $\mathcal{X}''$ is the span of the rows of the matrix
 \[
    W = \begin{pmatrix} 1&0&0&  1&1\\
                        0&1&0&  1&t\\
                        0&0&1&  1&t+t^3
    \end{pmatrix} \enspace .
 \]
 The link of the moduli space of tropical lines $\trop(L(X))$ 
 is  a $2$-dimensional polyhedral complex with $10$ vertices, $15$ rays, and $67$ maximal cones. 
 We restrict ourselves here to a brief description of the tropical lines that correspond to the vertices.
 The tropical plane $\mathcal{X}''$ has three vertices, say $a$, $b$, $c$, with three, four, and three rays attached, respectively. Moreover, it has two bounded edges, one between $a$ and $b$ and the other between $b$ and $c$.
 There is one tropical line contained in $\mathcal{X}''$ that has three vertices, and one that has a single vertex; both correspond to vertices of $\trop(L(X))$. 
 The one with three vertices (shown on the left of Figure~\ref{fig:valuated} in black) has the vertices $a$, $c$ and a third vertex in the $2$-dimensional cell containing $a$ and $c$. The position of this third vertex is fixed to be the component-wise maximum of $a$ and $c$. The single vertex of the other line is $b$, as illustrated on the left of Figure~\ref{fig:valuated} in red. The remaining eight vertices in $\trop(L(X))$ correspond to tropical lines with two vertices each. For two of these tropical lines, their bounded part is one of the two edges of $\mathcal{X}''$; see the right part of Figure~\ref{fig:valuated} in red for an example.
All the other six tropical lines have a vertex on a ray of $\mathcal{X}''$ and the other vertex is one of the three points $a$, $b$, $c$ to which the ray is not attached; Figure~\ref{fig:valuated} on the right shows an example in black. As before, the edge between these points is tropically convex, which determines the exact location of the point on the ray.
 
 To analyze this example we used the open source software \texttt{polymake} \cite{polymake} and the methods to compute tropical linear spaces which arise as valuations of linear spaces over Puiseux fractions; see \cite{JoswigLohoLorenzSchroeter:2016} and \cite{HampeJoswigSchroeter:2018} for details.
\end{example}

\begin{figure}[t]
    \centering
    \begin{tikzpicture}[
    	x  = {(0.5cm,.7cm)},
        y  = {(0cm,1cm)},
        z  = {(1cm,-.2cm)},
  edge/.style={line width=.8pt, line cap=round, blue},
  edge2/.style={line width=1.2pt, line cap=round, black, dotted},
  edge3/.style={line width=1.2pt, line cap=round, red, dotted},
  face/.style={draw=none, opacity=0.25, fill=blue},
  vertex/.style={draw = blue, fill = blue, line width=1.2pt},
  vertex2/.style={draw = black, fill = black, line width=1.2pt},
  vertex3/.style={draw = red, fill = red, line width=1.2pt},
  scale=1.7
  ]

  \coordinate (p1) at (0,0,0); 
  \coordinate (p2) at (-.9,-.3,-.9); 
  \coordinate (p0) at (-.3,-.3,-.3); 

  \coordinate (r0) at (0,1,0); 
  \coordinate (r1) at (1,-1,0); 
  \coordinate (r2) at (1,0,0); 
  \coordinate (r3) at (0,0,1); 
  \coordinate (r4) at (-1,0,-.5);
  
  \coordinate (l0) at (.3,-.9,-.3);
  \coordinate (l1) at (p2);
  \coordinate (l2) at ($(l0)+(r3)$);

  \coordinate (m0) at (p0);
  \coordinate (m1) at (p1);

  \draw[face] (p0) -- ($(p0)+(r1)$) -- ($(p0)+(r1)+(r2)$) -- ($(p0)+(r2)$) -- cycle;
  \draw[face] (p0) -- ($(p0)+(r1)$) -- ($(p0)+(r1)+(r3)$) -- ($(p0)+(r3)$) -- cycle;
  \draw[face] (p2) -- ($(p2)+(r0)$) -- ($(p2)+.7*(r0)+.7*(r1)$) -- ($(p2)+(r1)$) -- cycle;
  \draw[face] (p0) -- ($(p0)+(r1)$) -- ($(p2)+(r1)$) -- (p2) -- cycle;
  \draw[face] (p2) -- ($(p2)+(r4)$) -- ($(p2)+(r1)+(r4)$) -- ($(p2)+(r1)$) -- cycle;
  \draw[edge] (p0) -- ($(p0)+(r1)$);
  \draw[edge] (p2) -- ($(p2)+(r1)$);
        \draw[edge2] (l1) -- ($(l1)+(r1)$);
  \draw[edge] ($(p0)+(r2)$) -- (p0) -- ($(p0)+(r3)$);  
  \draw[edge] (p0) -- (p2) -- ($(p2)+(r4)$); 
        \draw[edge2] (l0) -- ($(l0)+(r2)$);
  \draw[face] (p0) -- ($(p0)+(r2)$)  -- ($(p1)+(r2)$) -- (p1) -- cycle;
  \draw[face] (p0) -- ($(p0)+(r3)$)  -- ($(p1)+(r3)$) -- (p1) -- cycle;
  \draw[face] (p0) -- ($(p0)+(r4)$) -- ($(p2)+(r4)$) -- (p2) -- cycle;
  \draw[face] (p1) -- ($(p1)+(r2)$) -- ($(p1)+(r2)+(r3)$) -- ($(p1)+(r3)$) -- cycle;
  \draw[face] (p0) -- ($(p0)+(r3)$) -- ($(p0)+(r4)$) -- cycle;
  \draw[face] (p0) -- ($(p0)+(r2)$) -- ($(p0)+1.2*(r2)+0.5*(r4)$) -- ($(p0)+(r4)$) -- cycle;
  \draw[face] (p1) -- ($(p1)+(r0)$) -- ($(p1)+(r0)+(r2)$) -- ($(p1)+(r2)$) -- cycle;
  \draw[face] (p1) -- ($(p1)+(r0)$) -- ($(p1)+(r0)+(r3)$) -- ($(p1)+(r3)$) -- cycle;
  \draw[face] (p0) -- (p1) -- ($(p1)+(r0)$) -- ($(p2)+(r0)$) -- (p2) -- cycle;
  \draw[face] (p2) -- ($(p2)+(r0)$) -- ($(p2)+(r0)+(r4)$) -- ($(p2)+(r4)$) -- cycle;
  
  \draw[edge] (p0) -- ($(p0)+(r4)$);
  \draw[edge] (p2) -- ($(p2)+(r4)$);
  \draw[edge] (p0) -- ($(p0)+(r4)$);
  \draw[edge] (p1) -- ($(p1)+(r2)$);
  \draw[edge] (p2) -- ($(p2)+(r4)$);
          \draw[edge3] (m0) -- ($(m0)+(r4)$);
  \draw[edge] (p1) -- ($(p1)+(r3)$);
  \draw[edge] ($(p2)+(r0)$) -- (p2) -- (p0) -- (p1) -- ($(p1)+(r0)$);

  \draw[edge3] (m1) -- ($(m1)+(r0)$);
  \draw[edge3] (m1) -- ($(m1)+(r3)$);
  \draw[edge3] (m1) -- ($(m1)+(r2)$);
  \draw[edge3] (m1) -- (m0); 
  \draw[edge3] (m0) -- ($(m0)+(r1)$);
 
  \draw[edge2] (l1) -- ($(l1)+(r4)$);
  \draw[edge2] (l1) -- ($(l1)+(r0)$);
  \draw[edge2] (l1) -- (l0) -- (l2); 
  
  \draw[vertex] ($(p0)$) circle (.6pt);
  \draw[vertex] ($(p1)$) circle (.6pt);
  \draw[vertex] ($(p2)$) circle (.6pt);

  \draw[vertex2] (l0) circle (.6pt);
  \draw[vertex2] (l1) circle (.6pt);

  \draw[vertex3] (m0) circle (.6pt);
  \draw[vertex3] (m1) circle (.6pt);

  \coordinate (trans) at (0,-.8,-4); 
  
  \coordinate (q1) at ($(0,0,0)+(trans)$); 
  \coordinate (q2) at ($(-.9,-.3,-.9)+(trans)$); 
  \coordinate (q0) at ($(-.3,-.3,-.3)+(trans)$); 

  \coordinate (L0) at ($(-.6,0,-.6)+(trans)$);
  \coordinate (L2) at ($(p1)+(trans)$);
  \coordinate (L1) at ($(p2)+(trans)$);

  \coordinate (M0) at (q0);

  \draw[face] (q0) -- ($(q0)+(r1)$) -- ($(q0)+(r1)+(r2)$) -- ($(q0)+(r2)$) -- cycle;
  \draw[face] (q0) -- ($(q0)+(r1)$) -- ($(q0)+(r1)+(r3)$) -- ($(q0)+(r3)$) -- cycle;
  \draw[face] (q2) -- ($(q2)+(r0)$) -- ($(q2)+.7*(r0)+.7*(r1)$) -- ($(q2)+(r1)$) -- cycle;
  \draw[face] (q0) -- ($(q0)+(r1)$) -- ($(q2)+(r1)$) -- (q2) -- cycle;
  \draw[face] (q2) -- ($(q2)+(r4)$) -- ($(q2)+(r1)+(r4)$) -- ($(q2)+(r1)$) -- cycle;
  \draw[edge] (q0) -- ($(q0)+(r1)$);
  \draw[edge] (q2) -- ($(q2)+(r1)$);
  \draw[edge] ($(q0)+(r2)$) -- (q0) -- ($(q0)+(r3)$);  
  \draw[edge] (q0) -- (q2) -- ($(q2)+(r4)$); 
        \draw[edge2] (L1) -- ($(L1)+(r1)$);
        \draw[edge3] (M0) -- ($(M0)+(r3)$);
        \draw[edge3] (M0) -- ($(M0)+(r2)$);
        \draw[edge3] (M0) -- ($(M0)+(r1)$);
  \draw[face] (q0) -- ($(q0)+(r2)$)  -- ($(q1)+(r2)$) -- (q1) -- cycle;
  \draw[face] (q0) -- ($(q0)+(r3)$)  -- ($(q1)+(r3)$) -- (q1) -- cycle;
  \draw[face] (q0) -- ($(q0)+(r4)$) -- ($(q2)+(r4)$) -- (q2) -- cycle;
  \draw[face] (q1) -- ($(q1)+(r2)$) -- ($(q1)+(r2)+(r3)$) -- ($(q1)+(r3)$) -- cycle;
  \draw[face] (q0) -- ($(q0)+(r3)$) -- ($(q0)+(r4)$) -- cycle;
  \draw[face] (q0) -- ($(q0)+(r2)$) -- ($(q0)+1.2*(r2)+0.5*(r4)$) -- ($(q0)+(r4)$) -- cycle;
  \draw[face] (q1) -- ($(q1)+(r0)$) -- ($(q1)+(r0)+(r2)$) -- ($(q1)+(r2)$) -- cycle;
  \draw[face] (q1) -- ($(q1)+(r0)$) -- ($(q1)+(r0)+(r3)$) -- ($(q1)+(r3)$) -- cycle;
  \draw[face] (q0) -- (q1) -- ($(q1)+(r0)$) -- ($(q2)+(r0)$) -- (q2) -- cycle;
        \draw[edge2] (L1) -- (L0) -- (L2);
        \draw[edge2] (L0) -- ($(L0)+0.8*(r0)$);
        \draw[edge3] (M0) -- ($(M0)+1.2*(r0)$);
  \draw[face] (q2) -- ($(q2)+(r0)$) -- ($(q2)+(r0)+(r4)$) -- ($(q2)+(r4)$) -- cycle;
  
  \draw[edge] (q0) -- ($(q0)+(r4)$);
  \draw[edge] (q2) -- ($(q2)+(r4)$);
  \draw[edge] (q0) -- ($(q0)+(r4)$);
  \draw[edge] (q1) -- ($(q1)+(r2)$);
  \draw[edge] (q2) -- ($(q2)+(r4)$);
  \draw[edge] (q1) -- ($(q1)+(r3)$);
  \draw[edge] ($(q2)+(r0)$) -- (q2) -- (q0) -- (q1) -- ($(q1)+(r0)$);

  \draw[edge3] (M0) -- ($(M0)+(r4)$);

  \draw[edge2] (L1) -- ($(L1)+(r4)$);
  \draw[edge2] (L2) -- ($(L2)+(r2)$);
  \draw[edge2] (L2) -- ($(L2)+(r3)$);
  
  \draw[vertex] ($(q0)$) circle (.6pt);
  \draw[vertex] ($(q1)$) circle (.6pt);
  \draw[vertex] ($(q2)$) circle (.6pt);

  \draw[vertex2] (L0) circle (.6pt);
  \draw[vertex2] (L1) circle (.6pt);
  \draw[vertex2] (L2) circle (.6pt);
  
  \draw[vertex3] (M0) circle (.6pt);

\end{tikzpicture}
    \caption{A projection of $\mathcal{X}''$ (with a self intersection) from Example~\ref{ex:valuated}, together with four tropical lines contained in it that correspond to vertices of the link of $\trop(L(X))$. }
    \label{fig:valuated}
\end{figure}
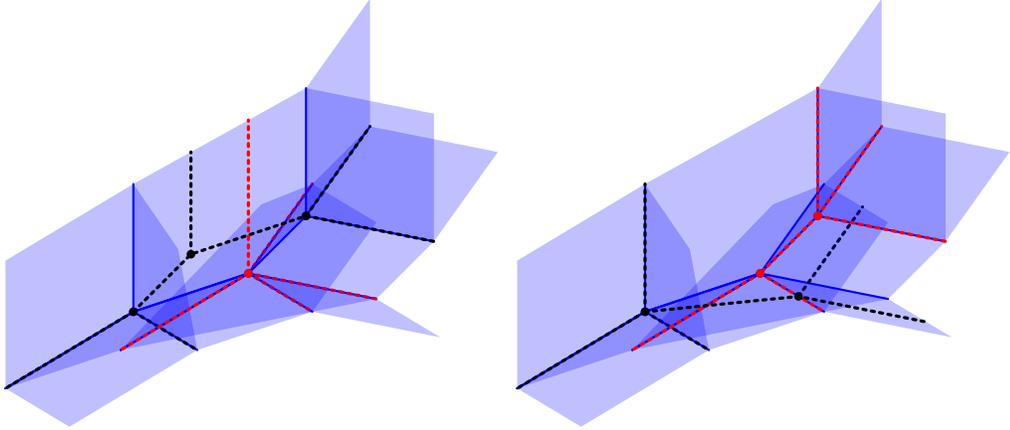
 
 \subsection{Background and connections}\label{sec-background}
 
In this subsection we briefly present additional motivation to study Question \ref{ques-linesinplanes} beyond the relative realizability problem, and we relate our study to other research topics in the literature.

 \subsubsection{Tropicalized moduli spaces of stable maps}
Tropical geometry has been successful in enumerative geometry, where an enumerative invariant is often treated as a $0$-dimensional Chow-cycle in a suitable moduli space parametrizing the objects to be counted. This fact, together with the close connection between tropical moduli spaces and their classical counterparts (which can for example be exploited to study topology \cite{CGP18} or compactifications \cite{Cap18}), triggered a large interest in tropical moduli spaces in the last years \cite{Mik06, GKM09, ACP15, CGP16, RSPW17, CMR16, Cha17, Gro16}. 

Often, moduli spaces of tropical curves of fixed genus $g$, possibly with $n$ marked points, $M^{\trop}_{g,n}$, are considered. 
If $g>0$, the moduli space $M^{\trop}_{g,n}$ can be given the structure of an abstract cone complex. For $g=0$, $M_{0,n}^{\trop}$ is a fan in a vector space, or, in other words, an embedded tropical variety. The embedding is obtained as a quotient of the Pl\"ucker embedding of the Grassmannian \cite{Tev07, GM07, SS04a}. 
It is equal to the Bergman fan of the matroid of the complete graph $K_{n-1}$ on $n-1$ vertices.

From the point of view of Gromov-Witten theory, it is most natural to study moduli spaces of stable maps into another variety $X$ next, since these are used to study the enumerative geometry of $X$. The case $g=0$, i.e.~the moduli space of stable rational maps, is most accessible; we restrict ourselves to this case for the purpose of this exposition. 
Consider the projective space $X=\mathbb{P}^n$. We may take the tropical (non-compact) model $\mathbb{R}^n$ and study moduli spaces of rational tropical stable maps of degree $d$ to $\mathbb{R}^n$, denoted by $M_{0,n}^{\trop}(\mathbb{R}^n,d)$. 
These spaces parametrize tuples $(\Gamma,x_1,\ldots,x_n,f)$, where $\Gamma$ is a metrized tree with some ends marked by $x_1,\ldots,x_n$, and $f \colon \Gamma \rightarrow \mathbb{R}^n$ is locally affine-linear such that $f(\Gamma)$ naturally obtains the structure of a tropical subvariety of $\mathbb{R}^n$ --- i.e., a weighted balanced graph.
The spaces $M_{0,n}^{\trop}(\mathbb{R}^n,d)$ are fans, and they are closely related to their classical counterparts \cite{GKM09, Ran17}.

The case of subvarieties of projective space is more challenging. The tropical analogue is a tropical subvariety $\mathcal{X}\subset \mathbb{R}^n$. Now realizability becomes an issue. Naively, we could let $M_{0,n}^{\trop}(\mathcal{X},d)$ parametrize $(\Gamma,x_1,\ldots,x_n,f)$ such that the tropical variety $f(\Gamma)$ is contained in $\mathcal{X}$.
But such a containment does not necessarily come from a tropicalization of a curve contained in a projective subvariety, so simply requiring containment is not sufficient from the purpose of realizability. In other words, there is a difference between moduli spaces of stable maps whose image is a tropical subvariety, and the tropicalized moduli space of stable maps.

Here, we care about realizability and thus focus on tropicalized moduli spaces of stable maps. Such tropicalized moduli spaces have been considered only for a few tractable situations, for example if $\mathcal{X}$ is a smooth rational curve \cite{GMO17}, or if $\mathcal{X}$ is a cubic surface \cite{GO17}. In the first case, $M_{0,n}^{\trop}(\mathcal{X},d)$ is a tropical variety (i.e., a weighted balanced polyhedral complex, but not a fan in general). The second moduli space is $0$-dimensional and the main task is to describe the tropical lines in the tropical cubic surface in question \cite{HJ17, CD19, PV19, PSS19, Gei19}.

With this paper, we add another class of tractable cases to the list: from the point of view of Gromov-Witten theory, the spaces $M^{\trop}_{0,n}(\mathcal{X},1)$ of stable maps of degree one for which $\mathcal{X}$ is a $2$-dimensional (realizable) tropical linear subspace can be treated with our approach.

\begin{corollary}
Let $\mathcal{X}$ be a realizable tropical plane in $\mathbb{R}^n$. The space $M^{\trop}_{0,0}(\mathcal{X},1)$ (viewed as the tropicalization of $M_{0,0}(X,1)$ for a generic lift $X$ of $\mathcal{X}$) equals the Bergman fan of relabeled Dilworth truncation $\widetilde{D}_{n-2}(U_{n,n})$ 
\end{corollary}

This corollary follows immediately from Corollary \ref{cor-DW}, as $M_{0,0}(X,1)$ is equal to the space of lines $L(X)$ in a generic plane $X$ whose tropicalization is $\mathcal{X}$.

\begin{example}\label{ex-m01}
Let $\mathcal{X}\subset \mathbb{R}^3$ be the Bergman fan of $U_{3,4}$, i.e., the standard tropical plane (see the left of Figure \ref{fig-U34}). Then $M^{\trop}_{0,0}(\mathcal{X},1)$ is the Bergman fan described in Example \ref{ex-standardplaneR3}. If we add one marked point and consider $M^{\trop}_{0,1}(\mathcal{X},1)$, we obtain a tropical line bundle over this Bergman fan, see Figure \ref{fig-linebundle}.
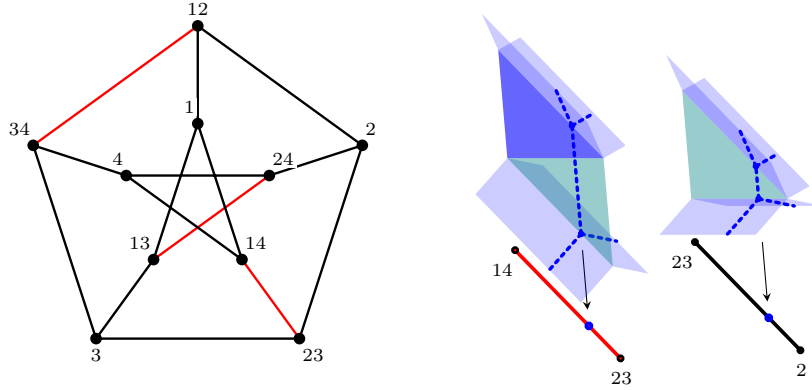
\begin{figure}
    \centering
    \begin{tikzpicture}[scale = 1,
                    edge/.style={black, line width=0.9pt, line cap=round},
                    edge2/.style={red, line width=0.9pt, line cap=round},
                    vertex/.style = {text=black, inner sep=2pt, white, draw=none}
                    ]

  \coordinate (p1) at (90:1); 
  \coordinate (p2) at (18:2.3); 
  \coordinate (p3) at (234:2.3); 
  \coordinate (p4) at (162:1); 

  \coordinate (p12) at (90:2.3); 
  \coordinate (p13) at (234:1); 
  \coordinate (p14) at (306:1); 
  \coordinate (p23) at (306:2.3); 
  \coordinate (p24) at (18:1); 
  \coordinate (p34) at (162:2.3);

  \draw[edge] (p34) -- (p3) -- (p23) -- (p2) -- (p12); 
  \draw[edge2] (p12) -- (p34); 
  \draw[edge2] (p13) -- (p24);
  \draw[edge] (p24) -- (p4) -- (p14) -- (p1) -- (p13); 
  \draw[edge] (p1) -- (p12); 
  \draw[edge] (p2) -- (p24); 
  \draw[edge] (p3) -- (p13); 
  \draw[edge] (p4) -- (p34); 
  \draw[edge2] (p14) -- (p23); 

  \fill (p1) circle [radius=2.2pt];
  \fill (p2) circle [radius=2.2pt];
  \fill (p3) circle [radius=2.2pt];
  \fill (p4) circle [radius=2.2pt];

  \fill (p12) circle [radius=2.2pt];
  \fill (p13) circle [radius=2.2pt];
  \fill (p14) circle [radius=2.2pt];
  \fill (p23) circle [radius=2.2pt];
  \fill (p24) circle [radius=2.2pt];
  \fill (p34) circle [radius=2.2pt];

  \node[above left, fill=white, inner sep=1pt] at ($(p1)+(0,0.1)$) {\tiny $1$};
  \node[above right, fill=white, inner sep=1pt] at ($(p2)+(0,0.1)$) {\tiny $2$};
  \node[below, fill=white, inner sep=1pt] at ($(p3)-(0,0.1)$) {\tiny $3$};
  \node[above left, fill=white, inner sep=1pt] at ($(p4)+(0,0.1)$) {\tiny $4$};

  \node[above, fill=white, inner sep=1pt] at ($(p12)+(0,0.1)$){\tiny $12$};
  \node[above left, fill=white, inner sep=1pt] at ($(p13)+(0,0.1)$){\tiny $13$}; 
  \node[above right, fill=white, inner sep=1pt] at ($(p14)+(0,0.1)$){\tiny $14$};
  \node[below right, fill=white, inner sep=1pt] at ($(p23)-(0,0.1)$){\tiny $23$};    
  \node[above right, fill=white, inner sep=1pt] at ($(p24)+(0,0.1)$){\tiny $24$};        
  \node[above left, fill=white, inner sep=1pt] at ($(p34)+(0,0.1)$){\tiny $34$};

 \begin{scope}[
        x  = {(0.076cm,-0.9)},
        y  = {(-1cm,0cm)},
        z  = {(.4cm,.6cm)},
  edge/.style={line width=.8pt, line cap=round, white},
  edge2/.style={line width=1.4pt, line cap=round, red},
  face/.style={draw=none, opacity=0.2, fill=blue},
  face2/.style={draw=none, opacity=0.4, fill=teal},
  face3/.style={draw=none, opacity=0.6, fill=blue},
  vertex/.style={draw = blue, line width=1.5pt},
  vertex2/.style={draw = black, line width=2pt},
  edge3/.style={line width=1.3pt, line cap=round, blue, dotted},
  scale=0.8
  ]
  \coordinate (trans) at (-.75,-6,0);
  \coordinate (x) at (0,-.8,0); 
  \coordinate (q0) at ($(0,0,0)+(trans)$); 
  \coordinate (q1) at ($(-1,1,1)+(trans)$); 
  \coordinate (q2) at ($(1,-1,1)+(trans)$); 
  \coordinate (q3) at ($(1,1,-1)+(trans)$); 
  \coordinate (q4) at ($(-1,-1,-1)+(trans)$); 

  \coordinate (q12) at ($(q1)+(q2)-(trans)$);
  \coordinate (q13) at ($(q1)+(q3)-(trans)$);
  \coordinate (q14) at ($(q1)+(q4)-(trans)$);
  \coordinate (q23) at ($(q2)+(q3)-(trans)$);
  \coordinate (q24) at ($(q2)+(q4)-(trans)$);
  \coordinate (q34) at ($(q3)+(q4)-(trans)$);

  \draw[edge2] ($(q0)-(x)+(1.7,0,0)$) -- ($(q23)+(x)+(1.7,0,0)$); 
  \draw[vertex2] ($(q0)-(x)+(1.7,0,0)$) circle (.6pt);
  \draw[vertex2] ($(q23)+(x)+(1.7,0,0)$) circle (.6pt);

  \draw[vertex2,blue] ($.7*(q23)+0.3*(q0)+0.4*(x)+(1.7,0,0)$) circle (.8pt);
  \draw[->,>=stealth] ($.7*(q23)+0.3*(q0)+0.4*(x)+(0.3,0,0)$) -- ($.7*(q23)+0.3*(q0)+0.4*(x)+(1.4,0,0)$);
  
  \draw[face] ($(q23)+(x)$) -- ($(q23)+0.4*(q2)-0.4*(trans)+(x)$) -- ($(q0)+0.4*(q2)-0.4*(trans)-(x)$) -- ($(q0)-(x)$) -- cycle;
            \draw[edge3] ($0.7*(q23)+0.3*(q0)+0.4*(x)$) -- ($0.7*(q23)+0.3*(q0)+0.4*(q2)-0.4*(trans)+0.4*(x)$); 
  \draw[face2] ($(q23)+(x)$) -- ($(q0)+(x)$) -- ($(q0)-(x)$) -- cycle;
  \draw[face3] ($(q0)+(x)$) -- ($(q14)-(x)$) -- ($(q0)-(x)$) -- cycle;
  \draw[face] ($(q23)+(x)$) -- ($(q23)+0.4*(q3)-0.4*(trans)+(x)$) -- ($(q0)+0.4*(q3)-0.4*(trans)-(x)$) -- ($(q0)-(x)$) -- cycle;
            \draw[edge3] ($0.7*(q23)+0.3*(q0)+0.4*(x)$) -- ($0.7*(q23)+0.3*(q0)+0.4*(q3)-0.4*(trans)+0.4*(x)$); 
  \draw[face] ($(q14)-(x)$) -- ($(q14)+0.7*(q4)-0.7*(trans)-(x)$) -- ($(q0)+0.7*(q4)-0.7*(trans)+(x)$) -- ($(q0)+(x)$) -- cycle; 
  \draw[face] ($(q14)-(x)$) -- ($(q14)+0.4*(q1)-0.4*(trans)-(x)$) -- ($(q0)+0.4*(q1)-0.4*(trans)+(x)$) -- ($(q0)+(x)$) -- cycle;
              \draw[edge3] ($0.7*(q0)+0.3*(q14)+0.4*(x)$) -- ($0.7*(q0)+0.3*(q14)+0.7*(q4)-0.7*(trans)+0.4*(x)$); 
              \draw[edge3] ($0.7*(q0)+0.3*(q14)+0.4*(x)$) -- ($0.7*(q0)+0.3*(q14)+0.4*(q1)-0.4*(trans)+0.4*(x)$); 
  

  \draw[edge3] ($0.7*(q23)+0.3*(q0)+0.4*(x)$) -- ($0.7*(q0)+0.3*(q14)+0.4*(x)$);

  \draw[vertex] ($0.7*(q23)+0.3*(q0)+0.4*(x)$) circle (.6pt);
  \draw[vertex] ($0.7*(q0)+0.3*(q14)+0.4*(x)$) circle (.6pt);
  
  \node[below, fill=white, inner sep=1pt] at ($(q23)+(x)+(1.9,0,0)$) {\tiny $23$}; 
  \node[below left, fill=white, inner sep=1pt] at ($(q0)-(x)+(1.9,0,0)$) {\tiny $14$};
  
\end{scope}

 \begin{scope}[
        x  = {(0.076cm,-0.9)},
        y  = {(-1cm,0cm)},
        z  = {(.4cm,.6cm)},
  edge/.style={line width=.8pt, line cap=round, white},
  edge2/.style={line width=1.4pt, line cap=round, black},
  face/.style={draw=none, opacity=0.2, fill=blue},
  face2/.style={draw=none, opacity=0.4, fill=teal},
  face3/.style={draw=none, opacity=0.7, fill=blue},
  vertex/.style={draw = blue, line width=1.5pt},
  vertex2/.style={draw = black, line width=2pt},
  edge3/.style={line width=1.3pt, line cap=round, blue, dotted},
  scale=.8
  ]

  \coordinate (trans) at (0,-9,0);
  
  \coordinate (x) at (0,-.8,0); 
  \coordinate (q0) at ($(0,0,0)+(trans)$); 
  \coordinate (q1) at ($(-1,1,1)+(trans)$); 
  \coordinate (q2) at ($(1,-1,1)+(trans)$); 
  \coordinate (q3) at ($(1,1,-1)+(trans)$); 
  \coordinate (q4) at ($(-1,-1,-1)+(trans)$); 

  \coordinate (q12) at ($(q1)+(q2)-(trans)$);
  \coordinate (q13) at ($(q1)+(q3)-(trans)$);
  \coordinate (q14) at ($(q1)+(q4)-(trans)$);
  \coordinate (q23) at ($(q2)+(q3)-(trans)$);
  \coordinate (q24) at ($(q2)+(q4)-(trans)$);
  \coordinate (q34) at ($(q3)+(q4)-(trans)$);

  \draw[edge2] ($(q0)-(x)+(.8,0,0)$) -- ($(q23)+(x)+(.8,0,0)$); 
  \draw[vertex2] ($(q0)-(x)+(.8,0,0)$) circle (.6pt);
  \draw[vertex2] ($(q23)+(x)+(.8,0,0)$) circle (.6pt);

  \draw[vertex2,blue] ($.7*(q23)+0.3*(q0)+0.4*(x)+(0.8,0,0)$) circle (.8pt);
  \draw[->,>=stealth] ($.7*(q23)+0.3*(q0)+0.4*(x)-(0.6,0,0)$) -- ($.7*(q23)+0.3*(q0)+0.4*(x)+(.5,0,0)$);
  
  \draw[face] ($(q0)+(x)$) -- ($(q0)+0.4*(q2)-0.4*(trans)+(x)$) -- ($(q0)+0.4*(q2)-0.4*(trans)-(x)$) -- ($(q0)-(x)$) -- cycle;
            \draw[edge3] ($(q0)+0.4*(x)$) -- ($(q0)+0.4*(q2)-0.4*(trans)+0.4*(x)$); 
  \draw[face2] ($(q0)+(x)$) -- ($(q14)-(x)$) -- ($(q0)-(x)$) -- cycle;
  \draw[face] ($(q0)+(x)$) -- ($(q0)+0.4*(q3)-0.4*(trans)+(x)$) -- ($(q0)+0.4*(q3)-0.4*(trans)-(x)$) -- ($(q0)-(x)$) -- cycle;
            \draw[edge3] ($(q0)+0.4*(x)$) -- ($(q0)+0.4*(q3)-0.4*(trans)+0.4*(x)$); 
  \draw[face] ($(q14)-(x)$) -- ($(q14)+0.7*(q4)-0.7*(trans)-(x)$) -- ($(q0)+0.7*(q4)-0.7*(trans)+(x)$) -- ($(q0)+(x)$) -- cycle; 
  \draw[face] ($(q14)-(x)$) -- ($(q14)+0.4*(q1)-0.4*(trans)-(x)$) -- ($(q0)+0.4*(q1)-0.4*(trans)+(x)$) -- ($(q0)+(x)$) -- cycle;
              \draw[edge3] ($0.7*(q0)+0.3*(q14)+0.4*(x)$) -- ($0.7*(q0)+0.3*(q14)+0.7*(q4)-0.7*(trans)+0.4*(x)$); 
              \draw[edge3] ($0.7*(q0)+0.3*(q14)+0.4*(x)$) -- ($0.7*(q0)+0.3*(q14)+0.4*(q1)-0.4*(trans)+0.4*(x)$);

  \draw[edge3] ($(q0)+0.4*(x)$) -- ($0.7*(q0)+0.3*(q14)+0.4*(x)$);

  \draw[vertex] ($(q0)+0.4*(x)$) circle (.6pt);
  \draw[vertex] ($0.7*(q0)+0.3*(q14)+0.4*(x)$) circle (.6pt);
  
  \node[below, fill=white, inner sep=1pt] at ($(q23)+(x)+(1,0,0)$) {\tiny $2$}; 
  \node[below left, fill=white, inner sep=1pt] at ($(q0)-(x)+(1,0,0)$) {\tiny $23$};
  
\end{scope}

\end{tikzpicture}
    \caption{The space $M^{\trop}_{0,1}(\mathcal{X},1)$ for $\mathcal X$ equal to the standard tropical plane in $\RR^3$ is a tropical line bundle over the Petersen graph. The three red edges parametrize lines passing through the apex of the plane~$\mathcal{X}$ (see Example~\ref{ex-standardplaneR3} and the right of Figure~\ref{fig-combtypes2}). Every point on such a line can be marked. The colors of the line bundle over the red edge indicates on which side the apex lies. }
    \label{fig-linebundle}
\end{figure}
\end{example}

\subsubsection{Tropicalized Fano schemes}
On the other hand, our main result also offers a new approach for the study of tropicalized Fano schemes, initiated in \cite{La18}. 
Fano schemes are moduli spaces parametrizing linear subspaces of algebraic varieties.
In \cite{La18}, the discrepancy between tropical Fano schemes (parametrizing not necessarily realizable tropical linear subspaces of a tropical variety) and tropicalized Fano schemes is discussed. The latter is in general strictly contained in the former.

Here, we contribute the study of tropicalized Fano schemes of linear subvarieties of dimension one higher, by describing them combinatorially.
In particular, we show that for generic $X$, the tropicalized Fano scheme is constant, and equal to the Bergman fan of a Dilworth truncation. With Theorem \ref{thm-main}, we offer a new perspective on tropicalized Fano schemes, and, in particular, on Example 3.4 in \cite{La18}; see our Example \ref{ex-genericity}.
 
\subsubsection{Tropicalized flag varieties}\label{subsec-flag}
Dressians are polyhedral fans that parametrize tropical linear spaces. They contain the tropicalized Grassmannians, which parametrize the realizable tropical linear spaces. Both objects offer a multitude of combinatorially intriguing research problems and results \cite{SS04a, HerrmannJensenJoswigSturmfels,FinkRincon,SchroeterThesis,OlartePanizzutSchroeter,SpeyerWilliams}. 

Flag varieties are subvarieties of products of Grassmannians, and it is natural to study their tropicalizations. Computationally, these problems are typically challenging. In \cite{BLMM}, the tropicalizations of the full flag varieties $\mathcal{F}_4$ (parametrizing flags of linear spaces $\{ P\subset L \subset H \subset \mathbb{P}^3\}$, where $P$ is a point, $L$ is a line, and $H$ is a plane) and $\mathcal{F}_5$ are computed. In \cite{BEZ}, Theorem 5.2.1 shows that the flag Dressian (i.e., the space parametrizing not necessarily realizable flags of tropical linear spaces) equals the tropicalized flag variety in these cases. 

As an example, the tropicalization of $\mathcal{F}_4$ (inside the torus) is a 6-dimensional fan in $\mathbb{R}^{11}$ with $78$ maximal cones (see Theorem 1 \cite{BLMM}).
With our study of tropicalized lines in planes, we offer a new perspective on this result, as we now explain. If no Pl\"ucker coordinates are allowed to be zero, all tropical planes in $\mathbb{R}^3$ are translations of the standard tropical plane $\mathcal X$ in $\mathbb{R}^3$, so they are uniquely determined by the position of their apex. Thus, inside the torus, the tropicalization of the flag variety parametrizing flags $\{ L \subset H \subset \mathbb{P}^3\}$ is exactly our generic tropicalized space of lines $L^{\trop}(\mathcal{X})$ times a $3$-dimensional lineality space. 
If we now want to add a point $P$ sitting on the line $L$ of the flag, we obtain, as in Example \ref{ex-m01}, a tropical line bundle over this space (see Figure \ref{fig-linebundle}). 
From this perspective, we can observe that we have a $6$-dimensional fan as the fan over the Petersen graph is two dimensional, the line bundle over it gives another dimension, and the lineality space which allows to translate the apex of $\trop(H)$ is three dimensional. This fan has $10\cdot 5+ 3\cdot 6=78$ maximal cones --- 5 cones over each of the black edges of the Petersen graph in Figure \ref{fig-linebundle} and 6 over each of the red edges. We expect that similar arguments can be used to compute and analyze further tropicalized flag varieties.

\section*{Appendix}\label{appendix}
\noindent This section is dedicated to the proof of Theorem~\ref{thm-incidenceimpliesPluecker} that each Pl\"ucker relation $R_{A,B}$ is contained in the saturation $(\mathcal{I}:\langle q_C \rangle^\infty)$.
The proof is somewhat technical, building on a double induction whose steps we spell out in the following lemmata. 

\begin{lemma}\label{lem-Equation1}
Let $A,B,C \subset [n]$ with $|A| = d-1$ and $|B| = |C| = d+1$, and suppose $a \in A\setminus B$.
For $i \in B\setminus A$ and $j\in C\setminus A$ with $j\neq i$, denote
\begin{align*}
\varphi_i &= \sgn{A}{B}{i} - \left(\sgn{A\setminus a \cup i}{C\cup a}{a}\right),\\
\psi_{j}& = \varphi_i + \sgn{A\setminus a \cup i}{C\cup a}{j} - \left(\sgn{A\setminus a\cup j}{B}{i}\right).
\end{align*} 
Then $(-1)^{\psi_j}$ is independent of $i$, and 	
\begin{align} \label{eq:in2pl}
\sum_{i\in B\setminus (A\setminus a)} (-1)^{\varphi_i} P_{B\setminus i} \cdot I_{A\setminus a \cup i, C\cup a} =
R_{A,B}\cdot q_C+\sum_{j\in C\setminus D}  (-1)^{\psi_{j}} 
R_{A\setminus a\cup j,B} \cdot q_{C\setminus j \cup a}.
\end{align}
\end{lemma}

\begin{proof}
First we show that $\psi_j$ modulo $2$ does not depend on $i$. 
To see this, notice that the summands of $\psi_j$ that depend on $i$ and do not cancel out 	are
\begin{align} \label{eq:noncancelingterms}
\Delta_j \coloneqq |A\cap [i]|-|A\setminus a \cup j \cap [i]|+ |A\setminus a \cup i \cap [j]| - |A\setminus a \cup i \cap [a]|.
\end{align}	
Assume first that $j<a$. 
We claim that
\begin{align} \label{eq:mod2psi}
\Delta_j  \equiv -|A\setminus a \cap \{j+1,\ldots,a\}| \ (\mathrm{mod}\ 2).
\end{align}
If $i$ is not between $j$ and $a$ then 
\begin{align*}
|A\setminus a \cup i \cap [j]| - |A\setminus a \cup i \cap [a]| = -|A\setminus a \cap \{j+1,\ldots,a\}|,
\end{align*} 
and if it is,
\begin{align*}
|A\setminus a \cup i \cap [j]| - |A\setminus a \cup i \cap [a]| = -|A\setminus a \cap \{j+1,\ldots,a\}| -1.
\end{align*}
Additionally, the first two summands in \eqref{eq:noncancelingterms} cancel if $i$ is not between $j$ and $a$ and 
contribute $-1$ if it is. 
Hence \eqref{eq:mod2psi} holds if $j < a$. 
The case $a < j$ is treated similarly, simply switching $a$ and $j$ in the argument.
	
It follows that for fixed $j$ the expression 
$ (-1)^{\psi_j} R_{A\setminus a\cup j,B} \cdot q_{C\setminus j \cup a}$ 
is well-defined. 
We can expand the Pl\"ucker relation and obtain a sum over all $i\in B\setminus (A\setminus a)$ with $i\neq j$. 
If, on the other hand, we expand for fixed $i\in B\setminus (A \setminus a)$ the polynomial 
$I_{D\setminus a \cup i, C\cup a}$ as a sum over all $j\in C\setminus D$, $j\neq i$, 
we deduce the claimed equality.
\end{proof}

\begin{lemma}\label{lem-rechnung1}
For $A, C \subset [n]$ with $|A| = d-1$ and $|C| = d+1$ we have $R_{A,C}\in (\mathcal{I}: \langle q_C \rangle ^\infty)$.
\end{lemma}

\begin{proof}
The proof is by induction on $|A\setminus C|$.
For the base case, suppose $|A\setminus C|=0$. 
Then $A\subset C$ and $R_{A,C}=0\in (\mathcal{I} : \langle q_C \rangle ^\infty)$.
The induction hypothesis is that $R_{A,C}\in (\mathcal{I}: \langle q_C \rangle ^\infty)$
for all $A$ with $|A\setminus C|<k$.
Now let $|A\setminus C|=k>0$. 
Then there exists $a\in A\setminus C$.
With that choice of $a$, Equation \eqref{eq:in2pl} with $B=C$ reads:
\begin{equation*}
\begin{aligned}
R_{A,C}\cdot q_C \ + \ 
 \sum_{j\in C\setminus A} (-1)^{\psi_j}\, R_{A\setminus a \cup j,C}\cdot q_{(C\cup j)\setminus a} \ = \
 \sum_{i\in C\setminus A} (-1)^{\varphi_i}\, P_{C\setminus i}\cdot I_{A\setminus a\cup i, C\cup a}.
\end{aligned}
\end{equation*}
The right-hand side is clearly in $(\mathcal{I} : \langle q_C \rangle ^\infty)$, as $\mathcal{I}$ is generated by the $I$'s. 
The sum over $j\in C\setminus A$ on the left hand side is also in 
$(\mathcal I: \langle Q_C \rangle ^\infty)$ by our induction hypothesis, 
since $|(A\setminus a\cup j)\setminus C|<|A\setminus C|=k$.
It follows that $R_{A,C}\in (\mathcal I: \langle Q_C \rangle ^\infty)$.
\end{proof}

\begin{lemma}\label{lem-rechnung2} Let $A, B, C \subset [n]$ with $|A| = d-1$, $|B| = |C| = d+1$, and $A\subset C$.
Let $b\in B\setminus C$ and $\beta = |[b]\cap B| + |[b]\cap C| + 1$. Furthermore, for $i \in B \setminus A$ and $j \in C \setminus B$, denote
$\varphi_i = |[i]\cap A|+|[i]\cap B|+|[b]\cap B\setminus i|+|[b]\cap C|$ and
$\psi_j = |[j]\cap B|+|[b]\cap B|+|[j]\cap C|+|[b]\cap C|$. Then
\begin{equation}\label{eq:moveB}
\begin{aligned}
R_{A,B}\cdot q_C + 
	\sum_{j\in C\setminus B} (-1)^{\psi_j}\, &R_{A,B\setminus b\cup j}\cdot q_{C\cup b\setminus j} = \\
&(-1)^\beta\, P_{B\setminus b}\cdot I_{A, C\cup b} + 
\sum_{\stackrel{i\in B\setminus A}{i\neq b}} (-1)^{\varphi_i}\, P_{A\cup i}\cdot 
I_{B\setminus i\setminus b, C\cup b}\enspace.
\end{aligned}
\end{equation}
\end{lemma}
\begin{proof}
Let $\nu_j = \beta + |[j]\cap A|+|[j]\cap (C\cup b)|$ and $\mu_{ij} = 
\varphi_i + |[j]\cap B\setminus i\setminus b|+|[j]\cap (C\cup b)|$.
Then for $i\neq b$ we have 
\begin{align}\label{eq:index}
\nu_i + \mu_{ii} \equiv 1+ 
|[i]\cap B|+|[i]\cap B\setminus i\setminus b|+|[b]\cap B|+|[b]\cap B\setminus i| \equiv 1  \ (\mathrm{mod}\ 2),
\end{align}
and if $i,j,b$ are distinct we also have
\begin{equation}\label{eq:index2}
\begin{aligned}
	\psi_j + |[i]\cap A| + |[i]\cap B\setminus b\cup j| &\equiv  \mu_{ij} && \ (\mathrm{mod}\ 2), &\text{ and }\\
	\psi_j + |[j]\cap A| + |[j]\cap B\setminus b\cup j| &\equiv  \nu_{j} && \ (\mathrm{mod}\ 2). 
\end{aligned}
\end{equation}
All three congruences are easily verified by distinguishing all six orderings of the variables $i$,$j$, and $b$.
	Furthermore, note that $b\in B\setminus C$ implies $b\not\in A$ as $A\subset C$.
Substituting the initial relations and distinguishing the two cases $j=b$ and $j\neq b$ for the index $j$ leads to the following equality:  
	\begin{align*}
		(-1)^\beta P_{B\setminus b}\, I_{A, C\cup b} \,&+
		\sum_{\stackrel{i\in B\setminus A}{i\neq b}} (-1)^{\varphi_i} P_{A\cup i}\, I_{B\setminus i\setminus b, C\cup b}
		& =\\ 
		P_{B\setminus b}
			\sum_{\stackrel{j\in C\setminus A}{\text{ or } j=b}} (-1)^{\nu_j} P_{A\cup j} \cdot q_{C\cup b \setminus j}
		\, &+ \,
		\sum_{\stackrel{i\in B\setminus A}{i\neq b}} P_{A\cup i} 
			\sum_{\stackrel{j\in C\setminus (B\setminus i)}{\text{ or } j=b}}(-1)^{\mu_{ij}} P_{B\setminus i \setminus b \cup j}\cdot q_{C\cup b\setminus j}
		&=\\
		(-1)^{|[b]\cap A|+|[b]\cap B|} P_{B\setminus b} P_{A\cup b} \cdot q_C &+
		\sum_{\stackrel{i\in B\setminus A}{i\neq b}} (-1)^{|[i]\cap A|+|[i]\cap B|} P_{A\cup i} P_{B\setminus i}\cdot q_{C} \,+\\
		P_{B\setminus b} \sum_{ j\in C\setminus A} (-1)^{\nu_j} P_{A\cup j} \cdot q_{C\cup b \setminus j}
		&+ \sum_{\stackrel{i\in B\setminus A}{i\neq b}} P_{A\cup i} 
			\sum_{j\in C\setminus (B\setminus i)} (-1)^{\mu_{ij}} P_{B\setminus i \setminus b \cup j}\cdot q_{C\cup b\setminus j} 
		\end{align*}
The first two terms sum to $R_{A,B}\cdot q_C$. Further changes of the summation order followed by a cancellation of terms using \eqref{eq:index} shows that the above equals:
		\begin{align*}
			\begin{array}{>{\displaystyle}c@{\hspace{0.1cm}}c@{\hspace{-.03cm}}>{\displaystyle}c@{\hspace{0.1cm}}c}
				R_{A,B}\cdot q_C  & + & 
		P_{B\setminus b} \!\!\sum_{ j\in C\setminus (A\cup B)}\!\! (-1)^{\nu_j} P_{A\cup j}\cdot q_{C\cup b \setminus j}
				& +\\
		\sum_{\stackrel{i\in B\setminus (A\cup C)}{i\neq b}}\!\! P_{A\cup i} \!\! 
			\sum_{ j\in C\setminus B } (-1)^{\mu_{ij}} P_{B\setminus i \setminus b \cup j} q_{C\cup b\setminus j}
				& + &
		\!\!\!\sum_{i\in (B\setminus A)\cap C} \!\! P_{A\cup i}\!\! 
			\sum_{ j\in C\setminus B } (-1)^{\mu_{ij}} P_{B\setminus i \setminus b \cup j} q_{C\cup b\setminus j} 
				& +\\
		P_{B\setminus b} \sum_{ j\in (C\setminus A)\cap B} (-1)^{\nu_j} P_{A\cup j}\cdot q_{C\cup b \setminus j}
				&+&
		\sum_{i\in (B\setminus A) \cap C} P_{A\cup i} (-1)^{\mu_{ii}} P_{B\setminus b}\cdot q_{C\cup b\setminus i}
			&\stackrel{\eqref{eq:index}}{=}\\
				R_{A,B}\cdot q_C  
				&+&
		P_{B\setminus b} \sum_{ j\in C\setminus (A\cup B)} (-1)^{\nu_j} P_{A\cup j}\cdot q_{C\cup b \setminus j}
				& +\\
		\sum_{ j\in C\setminus B } q_{C\cup b\setminus j} \!\!
			\sum_{\stackrel{i\in B\setminus (A\cup C)}{i\neq b}} (-1)^{\mu_{ij}} P_{A\cup i} P_{B\setminus i \setminus b \cup j}
				&+&
		\sum_{j\in C\setminus B} q_{C\cup b\setminus j} \!\!
			\sum_{ i\in (B\setminus A)\cap C } (-1)^{\mu_{ij}} P_{A\cup i} P_{B\setminus i \setminus b \cup j}&.
		\end{array}
		\end{align*}
	Combining the last two sums and then applying \eqref{eq:index2} leads to:
		\begin{align*}
			R_{A,B} q_C &+
		\sum_{ j\in C\setminus (A\cup B)} (-1)^{\nu_j} q_{C\cup b \setminus j}\,P_{A\cup j} P_{B\setminus b}
		+
		\sum_{j\in C\setminus B} q_{C\cup b\setminus j}
			\sum_{\stackrel{i\in B\setminus A}{i\neq b}} (-1)^{\mu_{ij}} P_{A\cup i} P_{B\setminus i \setminus b \cup j} & \!\! 
		\stackrel{\eqref{eq:index2}}{=}\\
			R_{A,B} q_C &+
		\sum_{j\in C\setminus B} q_{C\cup b\setminus j}
			\sum_{\stackrel{i\in (B\cup j)\setminus A}{i\neq b}} (-1)^{\psi_j + |[i]\cap A| + |[i]\cap B\setminus b\cup j|} P_{A\cup i} P_{B\setminus i \setminus b \cup j}\!\! &=\\
			R_{A,B}q_C &+
		\sum_{j\in C\setminus B} (-1)^{\psi_j} R_{A,B\setminus b\cup j} \, q_{C\cup b\setminus j}.
		\end{align*}
\end{proof}

Now we put everything together.

\begin{proof}[Proof of Theorem~\ref{thm-incidenceimpliesPluecker}]
The proof is by induction on $|A\cap C|$. 
For the induction base case, assume $|A\cap C|=d-1$, i.e., $A\subset C$. 

We now use another induction on $|B\cap C|$. 
For the second induction base case, assume $|B\cap C|=d+1$, i.e., $B=C$. 
Then the claim follows from Lemma \ref{lem-rechnung1}.
The second induction hypothesis is that 
$R_{A,B}\in (\mathcal I : \langle q_C \rangle ^\infty)$ for all $A\subset C$ and $B$ with $|B\cap C|>k$.
Now let $A\subset C$ and $B$ such that $|B\cap C|=k<d+1$. 
Then there exists $b\in B\setminus C$ and we can apply Lemma \ref{lem-rechnung2}. 
The right hand side of Equation \eqref{eq:moveB} is in the saturation $(\mathcal I : \langle q_C \rangle ^\infty)$. 
The sum over $j\in C\setminus B$ is in $(\mathcal I : \langle q_C \rangle ^\infty)$ by our hypothesis, 
since $|(B\setminus b\cup j)\cap C|>|B\cap C|=k$. 
It follows that $R_{A,B}\in (\mathcal I : \langle q_C \rangle ^\infty)$. 
Thus the inductive step for the second induction is complete, which proves the base case for the first induction.

The induction assumption is that for all $A$ and $B$ such that $|A\cap C|> k$ we have 
$R_{A,B}\in (\mathcal I : \langle q_C \rangle ^\infty)$.
Now let $A$ and $B$ be such that $|A\cap C|= k<d-1$. 
Then there exists $a\in A\setminus C$. 
We apply Lemma \ref{lem-Equation1}. 
The left hand side of Equation (\ref{eq:in2pl}) is in the ideal $(\mathcal I : \langle q_C \rangle ^\infty)$. 
The sum over $j\in C\setminus A$ on the right hand side is also in 
$(\mathcal I : \langle q_C \rangle ^\infty)$, since $|(A\setminus a \cup j)\cap C|>|A\cap C|=k$.
It follows that $R_{A,B}\in (\mathcal I : \langle q_C \rangle ^\infty)$.	
\end{proof}

\bibliographystyle {alpha}
\bibliography {JMRS}

\end{document}